%% file: VRPDA2.tex
\newif\ificml
\icmlfalse

\ificml
\documentclass{article}
\usepackage{icml2021}
\else
\documentclass[11pt,letter]{article} 
\usepackage[papersize={8.5in,11in},margin=0.7in]{geometry}
\fi

\input{math_commands.tex}

\ificml
\icmltitlerunning{Variance Reduction via Primal-Dual Accelerated Dual Averaging}
\else
\title{
Variance Reduction via Primal-Dual Accelerated Dual Averaging\\ for Nonsmooth Convex Finite-Sums\thanks{CS acknowledges support from the NSF award 2023239. JD acknowledges support from the Office of the Vice Chancellor for Research and Graduate Education at the University of Wisconsin–Madison with funding from the Wisconsin Alumni Research Foundation. SW acknowledges support from NSF Awards 1740707, 1839338, 1934612, and 2023239; Subcontract 8F-30039 from Argonne National Laboratory; and Award N660011824020 from the DARPA Lagrange Program.  }
}

\author{
  Chaobing Song, Stephen J. Wright and Jelena Diakonikolas\\
  Department of Computer Sciences\\
  University of Wisconsin-Madison\\
  \texttt{chaobing.song@wisc.edu, swright@cs.wisc.edu,  jelena@cs.wisc.edu}
}
\fi

\begin{document}
\ificml
\twocolumn[
\icmltitle{Variance Reduction via Primal-Dual Accelerated Dual Averaging\\ for Nonsmooth Convex Finite-Sums}

\icmlsetsymbol{equal}{*}

\begin{icmlauthorlist}
\icmlauthor{Chaobing Song}{uwm}
\icmlauthor{Stephen Wright}{uwm}
\icmlauthor{Jelena Diakonikolas}{uwm}
\end{icmlauthorlist}

\icmlaffiliation{uwm}{Department of Computer Sciences, University of Wisconsin-Madison, Madison, WI}

\icmlcorrespondingauthor{Chaobing Song}{chaobing.song@wisc.edu}

\icmlkeywords{variance reduction, dual averaging, primal-dual method, nonsmooth optimization}

\vskip 0.3in
]

\printAffiliationsAndNotice{\icmlEqualContribution} %
\else
\maketitle
\fi
\begin{abstract}
We study structured nonsmooth convex finite-sum optimization that appears widely in machine learning applications, including support vector machines and least absolute deviation. For the primal-dual formulation of this problem, we propose a novel algorithm called \emph{Variance Reduction via Primal-Dual Accelerated Dual Averaging (\vrpda)}. In the nonsmooth and general convex setting, \vrpda~has the overall complexity $O(nd\log\min \{1/\epsilon, n\} + d/\epsilon )$ in terms of the primal-dual gap, where $n$ denotes the number of samples, $d$ the dimension of the primal variables, and $\epsilon$ the desired accuracy. In the nonsmooth and strongly convex setting, the overall complexity of \vrpda~becomes $O(nd\log\min\{1/\epsilon, n\} + d/\sqrt{\epsilon})$ in terms of both the primal-dual gap and the distance between iterate and optimal solution. Both these results for \vrpda~improve significantly on state-of-the-art complexity estimates, which are $O(nd\log \min\{1/\epsilon, n\} + \sqrt{n}d/\epsilon)$ for the  nonsmooth and general convex  setting and $O(nd\log \min\{1/\epsilon, n\} + \sqrt{n}d/\sqrt{\epsilon})$ for the nonsmooth and strongly convex setting, in a much more simple and straightforward way. Moreover, both complexities are better than \emph{lower} bounds for general convex finite sums that lack the particular (common) structure that we consider. Our theoretical results are supported by numerical experiments, which confirm the competitive performance of \vrpda~compared to state-of-the-art.
\end{abstract}

\section{Introduction}

We consider large-scale regularized nonsmooth convex empirical risk minimization (ERM) of linear predictors in machine learning. Let $\vb_i\in\sR^d$, $i=1,2,\dotsc,n,$ be sample vectors with $n$ typically large; $g_i:\sR\rightarrow \sR$, $i=1,2,\dotsc,n,$ be possibly \emph{nonsmooth} convex loss functions associated with the linear predictor $\langle \vb_i, \vx\rangle$; and $\ell:\sR^d\rightarrow \sR$ be an extended-real-valued, $\sigma$-strongly convex ($\sigma\ge 0$) and  possibly nonsmooth regularizer that admits an efficiently computable proximal operator. 
The problem we study is 
\begin{equation}\tag{P}
\min_{\vx\in\sR^d} f(\vx) := g(\vx) + \ell(\vx) :=  \frac{1}{n}\sum_{i=1}^n g_i(\vb_i^T \vx) + \ell(\vx), \label{eq:prob}
\end{equation}
where $g(\vx):=\frac{1}{n}\sum_{i=1}^n  g_i(\vb_i^T \vx)$.  
Instances of the nonsmooth ERM problem \eqref{eq:prob} include $\ell_1$-norm and $\ell_2$-norm regularized support vector machines (SVM) and least absolute deviation. 
For practicality of our approach, we require in addition that the convex conjugates of the functions $g_i$, defined by $g_i^*(y_i) := \max_{z_i} (z_iy_i-g_i(z_i))$, admit efficiently computable proximal operators. (Such is true of the examples mentioned above.)
From the statistical perspective,  nonsmoothness in the loss function is essential for obtaining a model that is both tractable and robust. 
But from the optimization viewpoint, nonsmooth optimization problems are intrinsically more difficult to solve. On one hand, the lack of smoothness in $g$ precludes the use of black-box first-order information to obtain efficient methods. 
On the other hand, the use of structured composite optimization methods that rely on the proximal operator of $g$ is out of question here too,  because the proximal operator of the sum $\frac{1}{n}\sum_{i=1}^n g_i({\vb_i}^T \vx)$ may not be efficiently computable w.r.t.~$\vx$, even when the proximal operators of the individual functions $g_i(\cdot)$ are.

Driven by applications in machine learning, computational statistics, signal processing, and operations research, the nonsmooth problem \eqref{eq:prob} and its variants have been studied for more than two decades. 
There have been two main lines of work: deterministic algorithms that exploit the  underlying simple primal-dual structure to improve efficiency (i.e.,  dependence on the accuracy parameter $\epsilon$) and randomized algorithms that exploit the finite-sum structure to improve scalability (i.e., dependence on the number of samples $n$).    

\paragraph{Exploiting the primal-dual structure.}
A na\"{i}ve approach for solving \eqref{eq:prob} would be subgradient descent, which requires access to subgradients of $g(\vx)$ and  $\ell(\vx)$. 
To find a solution $\vx$ with  $f(\vx) - f(\vx^*)\le \epsilon$, where $\vx^*$ is an optimal solution of \eqref{eq:prob} and $\epsilon>0$ is the desired accuracy, the subgradient method requires $O(1/\epsilon^2)$ iterations for the nonsmooth convex setting. 
This complexity is high, but it is also the best possible  if we are only allowed to access  ``black-box'' information of function value and subgradient. 
To obtain improved complexity bounds, we must consider  approaches that exploit structure in \eqref{eq:prob}. 
To begin, we note that \eqref{eq:prob} admits an explicit and %
simple primal-dual %
reformulation: 
\begin{equation}\tag{PD}
\begin{gathered}\label{eq:primal-dual-deter}
\min_{\vx\in\sR^d}\big\{ f(\vx) = \max_{\vy\in\sR^n} L(\vx, \vy)  \big\},   \\
L(\vx,\vy) :=  \langle \mB\vx, \vy\rangle - g^*(\vy) + \ell(\vx),
\end{gathered}
\end{equation}
where $\mB = \frac{1}{n}[\vb_1, \vb_2, \ldots, \vb_n]^T$, $g^*(\vy) = \frac{1}{n}\sum_{i=1}^n g_i^* (y_i)$, and with the convex conjugate functions $g_i^*(\cdot)$ satisfying $  g_i(\vb_i^T\vx) = \max_{y_i}\{ y_i \langle \vb_i, \vx\rangle - g_i^*(y_i)\}.$ 
The nonsmooth loss %
$\vg(\vx)$ in \eqref{eq:prob} is thereby decoupled into a bilinear term $\langle\mB\vx, \vy\rangle$ and a separable function $g^*(\vy)$ that admits an efficiently computable proximal operator. 
Due to the possible nonsmoothness of $g(\vx),$  we can assume only that $L(\vx, \vy)$ is concave \emph{w.r.t.}~$\vy$---but not strongly concave. 
Therefore, Problem \eqref{eq:primal-dual-deter} is $\sigma$-strongly convex-(general) concave ($\sigma\ge 0$).

By adding a strongly convex regularizer to the dual variable of \eqref{eq:primal-dual-deter}, \citet{nesterov2005smooth} optimized a smoothed variant of \eqref{eq:prob} using %
acceleration, thus improving the complexity bound from $O(1/\epsilon^2)$ to $O(1/\epsilon)$. 
Later, Nemirovski and Nesterov, respectively, showed that extragradient-type methods such as mirror-prox  \citep{nemirovski2004prox}  and dual extrapolation \citep{nesterov2007dual} can obtain the same $O(1/\epsilon)$ complexity bound for \eqref{eq:primal-dual-deter} directly, without the use of smoothing or Nesterov acceleration. 
Extragradient-type methods need to perform updates twice per iteration, for both primal and dual variables. \citet{chambolle2011first} introduced an (extrapolated) primal-dual hybrid gradient (\pdhg) method to obtain the same $O(1/\epsilon)$ complexity with an extrapolation step on either the primal or dual variable rather than an extragradient step.
Thus, \pdhg~needs to update primal and dual variables just once per iteration. 
All three kinds of methods have been  extensively studied from different perspectives \citep{nesterov2005excessive,chen2017accelerated,tran2018smooth,song2020optimistic,diakonikolas2020efficient}. 
In the case of large $n$, the focus has been on  randomized variants with low per-iteration cost \citep{zhang2015stochastic, alacaoglu2017smooth, tan2018stochastic,chambolle2018stochastic, carmon2019variance,lei2019primal, devraj2019stochastic,alacaoglu2020random}.

\paragraph{Exploiting the finite-sum structure.} 
The deterministic methods discussed above have  per-iteration cost $O(nd)$, which can be prohibitively high when $n$ is large. 
There has been much work on randomized methods whose per-iteration cost is independent of $n$. 
To be efficient, the iteration count of such methods cannot increase too much over the deterministic methods.
A major development in the past decade of research %
has been the use of \emph{variance reduction} in randomized optimization algorithms, which reduces the per-iteration cost and improves the overall complexity. %
For the variant of Problem~\eqref{eq:prob} in which $g(\vx)$ is {\em smooth}, there exists a vast literature on developing efficient finite-sum solvers; see for example \citet{roux2012stochastic,johnson2013accelerating,lin2014accelerated, zhang2015stochastic, allen2017katyusha,zhou2018simple,lan2019unified, song2020variance}. In particular, the recent work of \citet{song2020variance} has proposed an algorithm called \emph{Variance Reduction via Accelerated Dual Averaging (\vrada)} that matches all  three lower bounds from  \citet{woodworth2016tight,hannah2018breaking} for the smooth and (general/ill-conditioned strongly/well-conditioned strongly) convex settings, using a  simple, unified algorithm description and convergence analysis. 
As discussed in \citet{song2020variance}, the efficiency, simplicity, and unification of \vrada~are due to randomizing \emph{accelerated dual averaging} rather than \emph{accelerated mirror descent} (as was done in \citet{allen2017katyusha}) and adopting a novel initialization strategy. 
The results of \citet{song2020variance} provide the main motivation for this work. %

When the loss function is nonsmooth, the classical variance reduction tricks such as \textsc{svrg} and \textsc{saga} \citep{johnson2013accelerating,defazio2014saga} are no longer available, as their effectiveness strongly depends on smoothness. 
A compromise proposed in \citet{allen2016optimal,allen2017katyusha} is to smoothen and regularize \eqref{eq:prob} and then apply existing finite-sum solvers, such as Katyusha. 
As shown by \citet{allen2017katyusha}, in the nonsmooth and general  convex setting, the resulting overall complexity\footnote{I.e., the number of iterations times the per-iteration cost.}  is improved from $O\big(\frac{nd}{\epsilon}\big)$ to $O\big(nd\log\min\{\frac{1}{\epsilon}, n\} + \frac{\sqrt{n}d}{\epsilon}\big)$; in the nonsmooth and strongly convex setting, it is improved from $O\big(\frac{nd}{\sqrt{\epsilon}}\big)$ to $O\big(nd\log \min\{\frac{1}{\epsilon}, n\} + \frac{\sqrt{n}d}{\sqrt{\epsilon}}\big)$. 
Both of these improved complexity results match the lower bounds of  \citet{woodworth2016tight} for general nonsmooth finite-sums when $\epsilon$ is small. 
However, the smoothing and regularization require tuning of additional  parameters, which complicates the algorithm implementation. %
Meanwhile, it is not clear whether the complexity can be further improved to take advantage of the additional ERM structure that is present in \eqref{eq:prob}. 

\begin{table*}[ht!]
\centering
\begin{threeparttable}[b]
\begin{small}
\caption{Overall complexity and per-iteration cost for solving  \eqref{eq:primal-dual-finite} in the $\sigma$-strongly convex-general concave setting ($\sigma \ge 0$). 
``---'' indicates that the corresponding result does not exist or is unknown.}
\tabcolsep=0.1cm %
\begin{tabular}{|c|c|c|c||c|}
\hline 
\multirow{2}{*}{Algorithm}      &       {General Convex}                &       {Strongly Convex}       & {Strongly Convex}               &     {Per-Iteration}            \\
                                &      (Primal-Dual Gap)           &    (Primal-Dual Gap)      &     (Distance to Solution)   &                Cost            \\
\hline
\textsc{rpd}             &                 \multirow{2}{*}{$O\big( \frac{n^{3/2}d}{\epsilon} \big)$}    &   \multirow{2}{*}{$O\big( \frac{n^{3/2}d}{\sqrt{\epsilon}} \big)$}   &  ---  & \multirow{2}{*}{$O(d)$}    \\   
\citet{dang2014randomized}           &                       &                           &    &    \\
\hline
{\textsc{smart-cd}}      &        \multirow{2}{*}{$O\big( \frac{nd}{\epsilon}\big)$}         &                \multirow{2}{*}{---}         &     \multirow{2}{*}{---}     &        \multirow{2}{*}{$O(d)$}        \\
\citet{alacaoglu2017smooth}  &         &      &   &\\
\hline
\citet{carmon2019variance}           &       $O\big(nd + \frac{\sqrt{nd(n+d)}\log(nd)}{\epsilon}\big)$       &        {---}         &     {---}    &        $O(n+d)$     \\            
\hline
\textsc{spdhg}   &    \multirow{2}{*}{$O\big( \frac{nd}{\epsilon}\big)$ }     &    \multirow{2}{*}{---}      &  \multirow{2}{*}{$O\big( \frac{nd}{\sigma\sqrt{\epsilon}}\big)$}\tnote{1}   &        \multirow{2}{*}{$O(d)$}    \\
\citet{chambolle2018stochastic} & & & &\\
\hline
\textsc{pure-cd}    &   \multirow{2}{*}{$O\big( \frac{n^2d}{\epsilon}\big)$}   &  \multirow{2}{*}{---}         &  \multirow{2}{*}{---}   &        \multirow{2}{*}{$O(d)$}                      \\ 
\citet{alacaoglu2020random} & & & &\\
\hline
{\vrpda}              &   \multirow{2}{*}{$O(nd \log {\min\{ \frac{1}{\epsilon}, n\}}%
+ \frac{d}{\epsilon})$}  & \multirow{2}{*}{$O(nd \log {\min\{ \frac{1}{\epsilon}, n\}}%
+ \frac{d}{\sqrt{\sigma\epsilon}})$}                &     \multirow{2}{*}{$O(nd \log \min\{ \frac{1}{\epsilon}, n\} + \frac{d}{\sigma\sqrt{\epsilon}})$}     &              \multirow{2}{*}{$O(d)$ }          \\
(\textbf{This Paper})  &         &      &   & \\  
\hline
\end{tabular}\label{tb:result}
\begin{tablenotes}
\item[1]It is only applicable when $\epsilon$ is small enough (see \citet[Theorem 5.1]{chambolle2018stochastic}). 
\end{tablenotes}

\end{small}
\vspace{0.1in}
\end{threeparttable}
\vspace{-0.6cm}
\end{table*}
 
For the nonsmooth ERM problem \eqref{eq:prob} considered here and its primal-dual formulation, the literature is much scarcer~\citep{dang2014randomized,alacaoglu2017smooth,chambolle2018stochastic,carmon2019variance,latafat2019new, fercoq2019coordinate, alacaoglu2020random}. 
All of the existing methods target \eqref{eq:primal-dual-deter} directly and focus on extending the aforementioned deterministic algorithms to this case.
As sampling one element of the finite sum from \eqref{eq:prob} is reduced to sampling one dual coordinate in \eqref{eq:primal-dual-deter}, all of these methods can be viewed as coordinate variants of the deterministic counterparts. 
For convenience, we explicitly rewrite \eqref{eq:primal-dual-deter} in the following finite-sum primal-dual %
form: 
\begin{equation}\tag{FS-PD} \label{eq:primal-dual-finite}
\begin{gathered}
\min_{\vx\in\sR^d}\max_{\vy\in\sR^n}L(\vx, \vy),\\
L(\vx, \vy) = \frac{1}{n}\sum_{i=1}^n \left(  y_i\langle \vb_i,  \vx\rangle - g_i^*(y_i)\right) + \ell(\vx). %
\end{gathered}
\end{equation}
\paragraph{Existing approaches.} %
 Table \ref{tb:result} compares \vrpda~to existing randomized algorithms for solving~\eqref{eq:primal-dual-finite} in terms of the overall complexity and per-iteration cost under the setting of uniform sampling and general data matrix (\emph{i.e.,} the data matrix is not necessarily sparse).  %
The competing algorithms (\textsc{rpd}, \textsc{smart-cd}, \textsc{spdhg}, and \textsc{pure-cd}) attain $O(d)$ per-iteration cost, but have overall complexity no better than that of the deterministic algorithms in both the nonsmooth and general/strongly convex settings. 
Meanwhile, the algorithms of \citet{carmon2019variance} %
attain $O(n+d)$ per-iteration cost with improved dependence on the dimension $d$ when $n\ge d.$ 
However, the overall dependence on the dominant term $n$ is still not improved, which raises the question of whether it is even possible to simultaneously achieve the low $O(d)$ per-iteration cost and reduce the overall complexity compared to the deterministic algorithms. 
Addressing this question is the main contribution of our work. 

\paragraph{Our contributions.} We propose %
the \vrpda~algorithm 
for \eqref{eq:primal-dual-finite} in the $\sigma$-strongly convex-general concave setting  $(\sigma \ge 0)$, which corresponds to the nonsmooth and $\sigma$-strongly convex setting of \eqref{eq:prob}  $(\sigma \ge 0)$. For both settings with $\sigma=0$ and $\sigma>0$, \vrpda~has $O(d)$ per-iteration cost and  significantly improves the  best-known overall complexity  results in a unified and simplified way. 
As shown in Table \ref{tb:result}, to find an $\epsilon$-accurate solution in terms of the primal-dual gap, the overall complexity of \vrpda~is 
\begin{align*}
\begin{cases}
O\big( n d\log\big({\min\{ \frac{1}{\epsilon}, n\}}\big)  + \frac{d}{\epsilon}\big),  & \text{ if }  \sigma = 0, \\ 
O\big( n d\log \big({\min\{ \frac{1}{\epsilon}, n\}}\big)  + \frac{d}{\sqrt{\sigma\epsilon}}\big), &\text{ if }  \sigma > 0,
\end{cases}
\end{align*}
which is significantly better than any of the existing results for \eqref{eq:primal-dual-finite}. In particular, we only need $O(nd\log n)$ overall cost  to  attain an $\epsilon$-accurate solution with $\epsilon = \Omega(\frac{1}{n\log(n)})$.
Meanwhile,  when $\epsilon$ is sufficiently small compared to $1/n$, so that the second term in the bound becomes dominant, the overall complexity ($O(\frac{d}{\epsilon})$ for $\sigma=0$ and $O(\frac{d}{\sqrt{\sigma\epsilon}})$ for $\sigma>0$) is independent of $n$, thus showing a $\Theta(n)$ improvement compared to the deterministic algorithms. 
To the best of our knowledge, even for smooth $g_i$'s, %
the improvement of existing algorithms is at most %
$\Theta(\sqrt{n})$ and is attained by accelerated variance reduction methods such as {Katyusha}~\cite{allen2017katyusha} and \textsc{vrada}~\cite{song2020variance}.

\paragraph{Comparison to lower bounds.} 
What makes our results particularly surprising is that they seem to contradict the iteration complexity lower bounds for composite objectives, which are $\Omega(n + \frac{\sqrt{n}}{\epsilon})$ for nonsmooth and general convex objectives and $\Omega(n + \sqrt{\frac{n}{\sigma \epsilon}})$ for nonsmooth and  $\sigma$-strongly convex objectives~\citep{woodworth2016tight}. In \cite[Section 5.1]{woodworth2016tight}, 
 the hard instance for proving the lower bounds has the form $f(x) = \frac{1}{n}\sum_{i=1}^n f_i(x)$---but each $f_i$ is a sum of $k+1$ ``simple’’ terms, each having the form of our $g_i$’s.
The complexity in~\cite{woodworth2016tight} for this hard instance is enabled by hiding the individual vectors corresponding to each simple term, %
an approach that is typical for oracle lower bounds.
In their example, $k = \Theta(\frac{1}{\sqrt{n}\epsilon})$, so the total number of simple terms (of the form $g_i$  in our framework) is  $nk = \Theta(\frac{\sqrt{n}}{\epsilon})$, which leads to the second term in the lower bound. 
(The first $\Omega(n)$ term in this lower bound comes from setting $\epsilon = O(\frac{1}{\sqrt{n}})$.)
Applying our upper bound for iteration complexity to this hard case, we replace $n$ by $nk = \Theta(\frac{\sqrt{n}}{\epsilon})$ to obtain $O(\frac{\sqrt{n}}{\epsilon}\log(\frac{\sqrt{n}}{\epsilon}))$---higher than the \citet{woodworth2016tight} lower bound would be if we were to replace $n$ by $nk$. 
Thus, there is no contradiction.

Remarkably, our upper bounds show that use of the finite-sum primal-dual formulation \eqref{eq:primal-dual-finite} can lead not only to improvements in efficiency (dependence on $\epsilon$), as in \citet{nesterov2005smooth}, but also scalability (dependence on $n$).   
As the ERM problem \eqref{eq:prob} is one of the main motivations for convex finite-sum solvers, it would be interesting to characterize the complexity of %
the problem class \eqref{eq:prob} from the aspect of oracle lower bounds and determine whether \vrpda~attains optimal oracle complexity. 
(We conjecture that it does, at least for small values of $\epsilon$.) 
Since the primary focus of the current paper is on algorithms, we leave the study of lower bounds for future research.

\paragraph{Our techniques.} 
Our \vrpda~algorithm is founded on a new deterministic algorithm \emph{Primal-Dual Accelerated Dual Averaging (\pda)} for  \eqref{eq:primal-dual-deter}. 
Similar to \textsc{pdhg} \citep{chambolle2011first}, \pda~is a primal-dual method with extrapolation on the primal or dual variable. 
However, unlike \textsc{pdhg}, which is based on mirror-descent-type updates (\emph{a.k.a.}~agile updates \citep{allen2017linear}), \pda~performs updates of dual averaging-style \citep{nesterov2015universal} (\emph{a.k.a.}~lazy mirror-descent updates \citep{hazan2016introduction}). 

Our analysis is based on the classical estimate sequence technique, but with a novel design of the estimate sequences that requires careful coupling of primal and dual portions of the gap; see Section~\ref{sec:pdada} for a further discussion. 
The resulting argument allows us to use  a unified parameter setting and convergence analysis for \pda~in all the (general/strongly) convex-(general/strongly) concave settings. 
Thus, by building on \pda~rather than \textsc{pdhg}, the design and analysis of \vrpda~is unified over the different settings and also significantly simplified. 
Moreover, the dual averaging framework allows us to use a novel initialization strategy inspired by the \textsc{vrada} algorithm \citep{song2020variance}, which is key to cancelling the randomized error of order $n$ in the main loop and obtaining our improved results from Table~\ref{tb:result}.  
It is worth noting that although \pda~can be used in all the (general/strongly) convex-(general/strongly) concave settings,  \vrpda~is only applicable to the (general/strongly) convex-general concave settings, which correspond to the nonsmooth and (general/strongly) convex settings of \eqref{eq:prob}. 
Study of \vrpda~in the (general/strongly) convex-strongly concave settings %
is deferred to future research.

\section{Notation and Preliminaries}\label{sec:prelims}

Throughout the paper, we use $\|\cdot\|$ to denote the Euclidean norm. In the case of matrices $\mB,$ $\|\mB\|$ is the standard operator norm defined by $\|\mB\| := \max_{\vx \in \sR^d, \, \|\vx\|\le 1}  \|\mB\vx\|$.

In the following, we provide standard definitions and properties that will be used in our analysis. We start by stating the definition of strongly convex functions that captures both strong  and general convexity, allowing us to treat both cases in a unified manner for significant portions of the analysis.
We use $\bar{\sR} = \sR \cup \{+\infty\}$ to denote the extended real line.

\begin{definition}\label{def:str-cvx}
Given $\sigma \geq 0$, we say that a function $f: \sR^d \to \bar{\sR}$ is $\sigma$-strongly convex, if $\forall \vx, \hat\vx \in \sR^d$, and all $\alpha \in (0, 1)$
\ificml
\begin{align*}
    f((1-\alpha)\vx + \alpha \hat\vx) \leq\;& (1-\alpha)f(\vx) + \alpha f(\hat\vx)\\
    &- \frac{\sigma}{2}\alpha(1-\alpha)\|\hat\vx - \vx\|^2. 
\end{align*}
\else
\begin{align*}
    f((1-\alpha)\vx + \alpha \hat\vx) \leq (1-\alpha)f(\vx) + \alpha f(\hat\vx) - \frac{\sigma}{2}\alpha(1-\alpha)\|\hat\vx - \vx\|^2. 
\end{align*}
\fi
When $\sigma = 0,$ we say that $f$ is (general) convex.
\end{definition}

When $f$ is subdifferentiable at $\vx$ and $\vg_f(\vx) \in \partial f(\vx)$ is any subgradient of $f$ at $\vx$, where  $\partial f(\vx)$ denotes the subdifferential set (the set of all subgradients) of $f$ at $\vx,$ then strong convexity implies that for all $\hat\vx \in \sR^d,$ we have
$$
    f(\hat\vx) \geq f(\vx) + \innp{\vg_f(\vx), \hat\vx - \vx} + \frac{\sigma}{2}\|\hat\vx - \vx\|^2.
$$

Since we work with general nonsmooth convex functions $f$, we require that their \emph{proximal operators}, defined %
as solutions 
to problems of the form
$
    \min_{\vx \in \sR^d} \big\{f(\vx) + \frac{1}{2\tau}\|\vx - \vxh\|^2\big\}
$ 
are efficiently solvable for any $\tau > 0$ and any $\vxh \in \sR^d.$

\paragraph{Problem definition.} 
As discussed in the introduction, our focus is on Problem~\eqref{eq:primal-dual-deter} 
under the following assumption. 

\begin{assumption} \label{ass:general}
$g^*(\vy)$ is proper, l.s.c., and $\gamma$-strongly convex ($\gamma\ge 0$); %
$\ell(\vx)$ is proper, l.s.c., and $\sigma$-strongly convex ($\sigma\ge 0$);
the proximal operators of %
$g^*$  and $\ell$ can be computed efficiently; and $\|\mB\| = R$ for some $R \in (0, \infty)$.  
\end{assumption}

Observe that since $g^*$ and $\ell$ are only assumed to be proper, l.s.c., and (strongly) convex, they can contain indicators of closed convex sets in their description. 
Thus, the problem class that satisfies Assumption~\ref{ass:general} contains constrained optimization. 
We use $\gX$ and $\gY$ to denote the domains of $\ell$ and $g^*$, respectively, defined by $\gX = \mathrm{dom}(\ell) = \{\vx: \ell(\vx) < \infty\}$, $\gY = \mathrm{dom}(g^*) = \{\vy: g^*(\vy) < \infty\}$. When $\gX, \gY$ are bounded, we use $D_{\gX}, D_{\gY}$ to denote their diameters, $D_{\gX} = \max_{\vx, \vu\in\gX} \|\vx - \vu\|$,  $D_{\gY} = \max_{\vy, \vv\in\gY} \|\vy - \vv\|$.

Note that Assumption~\ref{ass:general} does not enforce a finite-sum structure of $g^*$ (and $g$). Thus, for the results that utilize variance reduction, we will make a further assumption.
\begin{assumption} \label{ass:vr}
$g^*(\vy) = \frac{1}{n}\sum_{i=1}^n g_i^*(y_i),$ where each $g_i^* (y_i)$ is  convex and has an efficiently computable proximal operator. Further, $\|\vb_i\|\le R'$, for all $i \in \{1, \dots, n\}$.  
\end{assumption}
Recall that $\mB = \frac{1}{n}[\vb_1, \vb_2, \ldots, \vb_n]^T$. Observe that $R = \|\mB\| \leq %
\frac{1}{n}\Big(\sum_{i=1}^n \|\vb_i\|^2\Big)^{1/2} \leq \frac{1}{n}\sum_{i=1}^n \|\vb_i\| \leq R'.$ %

Observe further that, under Assumption~\ref{ass:vr}, $g^*(\vy)$ is separable over its coordinates. As a consequence, the domain $\gY$ of $g^*$ can be expressed as the Cartesian product of $\mathrm{dom}(g^*_i)$. 
This structure is crucial for variance reduction, as the algorithm in this case relies on performing coordinate descent updates over the dual variables $\vy$. 

\paragraph{Primal-dual gap.}

Given $\vx \in \sR^d$, the \emph{primal value} of the problem~\eqref{eq:primal-dual-deter} is
$
    P(\vx) = \max_{\vv\in\sR^n} \, L(\vx, \vv).
$ 
Similarly, the \emph{dual value}~\eqref{eq:primal-dual-deter} is defined by
$
    D(\vy) = \min_{\vu\in\sR^d} \, L(\vu, \vy). 
$ 
Given a primal-dual pair $(\vx, \vy) \in \sR^d \times \sR^n,$ primal-dual gap is then defined by
\ificml
$ \Gap(\vx, \vy) = P(\vx) - D(\vy) = \max_{(\vu, \vv) \in \sR^d \times \sR^n}\Gap^{\vu, \vv}(\vx, \vy),$ 
\else
\begin{equation}\notag
    \Gap(\vx, \vy) = P(\vx) - D(\vy) = \max_{(\vu, \vv) \in \sR^d \times \sR^n}\Gap^{\vu, \vv}(\vx, \vy),  
\end{equation}
\fi
where we define
\begin{equation} \label{eq:Guv}
\Gap^{\vu, \vv}(\vx, \vy) = L(\vx, \vv) - L(\vu, \vy). %
\end{equation}
Observe that, by definition of $P(\vx)$ and $D(\vy),$ the maximum of $\Gap^{\vu, \vv}(\vx, \vy)$ for fixed $(\vx, \vy)$ is attained when $(\vu, \vv) \in \gX \times \gY,$ so we can also write $\Gap(\vx, \vy) = \max_{(\vu, \vv) \in \gX \times \gY}\Gap^{\vu, \vv}(\vx, \vy)$.

For our analysis, it is useful to work with the relaxed gap $\Gap^{\vu, \vv}(\vx, \vy).$ In particular, to bound the primal-dual gap $\Gap(\vx, \vy)$ for a candidate solution pair $(\vx, \vy)$ constructed by the algorithm, we first bound $\Gap^{\vu, \vv}(\vx, \vy)$ for arbitrary $(\vu, \vv) \in \gX \times \gY.$ The bound on $\Gap(\vx, \vy)$ then follows by taking the supremum of  $\Gap^{\vu, \vv}(\vx, \vy)$ over $(\vu, \vv) \in \gX \times \gY.$ In general, $\Gap(\vx, \vy)$ can be bounded by a finite quantity only when $\gX, \gY$ are compact~\cite{nesterov2005smooth,ouyang2018lower}. If either of $\gX, \gY$ is unbounded, to provide meaningful results and similar to~\citet{chambolle2011first}, we assume that an optimal primal-dual pair $(\vx^*, \vy^*)$ for which $\Gap(\vx^*, \vy^*) = 0$ exists, and bound the primal-dual gap in a ball around $(\vx^*, \vy^*).$   

\paragraph{Auxiliary results.} Additional auxiliary results on  growth of sequences that are needed when establishing convergence rates in our results are provided in Appendix~\ref{appx:growth-of-seqs}.

\section{Primal-Dual Accelerated Dual Averaging %
}\label{sec:pdada}

In this section, we provide the \pda~algorithm for solving Problem~\eqref{eq:primal-dual-deter} under Assumption~\ref{ass:general}. 
The results in this section provide the basis for our results in Section~\ref{sec:vrpdada}  for the finite-sum primal-dual setting. 

\pda~is described in Algorithm~\ref{alg:pda}. 
Observe that the points $\vu$, $\vv$ in the definitions of estimate sequences $\phi_k(\vx), \psi_k(\vy)$ do not play a role in the definitions of $\vx_k$, $\vy_k,$ as the corresponding $\argmin$s are independent of $\vu$ and $\vv$. 
They appear in the definitions of $\phi_k(\vx), \psi_k(\vy)$ only for the convenience of the convergence analysis; the algorithm itself can be stated without them.

\begin{algorithm}[htb!]
\caption{ Primal-Dual Accelerated Dual Averaging (\pda)}\label{alg:pda}
\begin{algorithmic}[1]
\STATE \textbf{Input: } $(\vx_0, \vy_0)\in {\gX \times \gY}, (\vu, \vv)\in \gX \times \gY, 
\sigma \geq 0, \gamma \geq 0, \|\mB\|=R>0, K.$ %
\STATE $a_0 = A_0 = 0.$
\STATE $ \vx_0 = \vx_{-1}\in\sR^d, \vy_0\in\sR^n.$
\STATE $\phi_0(\cdot) = \frac{1}{2}\|\cdot - \vx_0\|^2, \psi_{0}(\cdot) = \frac{1}{2}\|\cdot - \vy_0\|^2$. 
\FOR{$k = 1,2,\ldots, K$} 
\STATE $a_k = \frac{ \sqrt{(1+\sigma A_{k-1})(1+\gamma A_{k-1})}}{\sqrt{2}R}, A_k = A_{k-1} + a_k$.
\STATE $\bar{\vx}_{k-1} =  \vx_{k-1} + \frac{a_{k-1}}{a_{k}}(\vx_{k-1} - \vx_{k-2}).$
\STATE $\vy_{k} = \argmin_{\vy\in\sR^n}\{  \psi_k(\vy) =  \psi_{k-1}(\vy) + a_k (\langle - \mB\bar{\vx}_{k-1}, \vy - \vv \rangle +  g^*(\vy))
\}.$
\STATE $\vx_{k} = \argmin_{\vx\in\sR^d}\{  \phi_k(\vx) =  \phi_{k-1}(\vx) + a_k (\langle {\vx} - \vu, \mB^T\vy_k\rangle+\ell(\vx))
\}.$
\ENDFOR
\STATE \textbf{return } $\tilde{\vy}_K = \frac{1}{A_K}\sum_{k=1}^{K}a_k\vy_k$, $\tilde{\vx}_K = \frac{1}{A_K}\sum_{k=1}^{K}a_k\vx_k.$
\end{algorithmic}
\end{algorithm}

We now outline the  main technical ideas in the \pda~algorithm. 
To bound the relaxed notion of the primal-dual gap $\Gap^{\vu, \vv}(\vxt_k, \vyt_k)$ discussed in Section~\ref{sec:prelims}, we use estimate sequences $\phi_k(\vx)$ and $\psi_k(\vy)$ defined in the algorithm. 
Unlike the classical estimate sequences used, for example, in~\citet{nesterov2005smooth}, these estimate sequences do not directly estimate the values of the primal and dual, %
but instead contain additional bilinear terms, which are crucial for forming an intricate coupling argument between the primal and the dual that leads to the desired convergence bounds. 
In particular, the bilinear term in the definition of $\psi_k$ is defined w.r.t.~an extrapolated point $\bar{\vx}_{k-1}$. 
This extrapolated point is not guaranteed to lie in the domain of $\ell$, but because this point appears only in bilinear terms, we never need to evaluate either $\ell$ or its subgradient at $\bar{\vx}_{k-1}.$ 
Instead, the extrapolated point plays a role in cancelling  error terms that appear when relating the estimate sequences to $\Gap^{\vu, \vv}(\vxt_k, \vyt_k)$. %

Our main technical result for this section concerning the convergence of \pda~is summarized in the following theorem. 
The proof of this result and supporting technical results 
are  provided in Appendix~\ref{appx:omitted-proofs-general}.

\begin{restatable}{theorem}{thmgeneral} \label{thm:general}
Under Assumption \ref{ass:general}, for Algorithm \ref{alg:pda}, 
we have, $\forall (\vu, \vv)\in {\gX \times \gY}$ 
and $k\ge 1,$ 
\begin{equation*}
    \Gap^{\vu, \vv}(\vxt_k, \vyt_k) \leq \frac{\|\vu - \vx_0\|^2 + \|\vv - \vy_0\|^2}{2 A_k}.
\end{equation*}
Further, if $(\vx^*, \vy^*)$ is a primal-dual solution to~\eqref{eq:primal-dual-deter}, then
\ificml
\begin{equation}\label{eq:pda2-bnded-diam}
\begin{aligned}
    (1+\sigma A_k)&\|\vx_k - \vx^*\|^2 + \frac{1+\gamma A_k}{2}\|\vy_k - \vy^*\|^2\\
    &\leq \|\vx_0 - \vx^*\|^2 + \|\vy_0 - \vy^*\|^2.
\end{aligned}
\end{equation}
\else
\begin{equation}\label{eq:pda2-bnded-diam}
    (1+\sigma A_k)\|\vx_k - \vx^*\|^2 + \frac{1+\gamma A_k}{2}\|\vy_k - \vy^*\|^2 \leq \|\vx_0 - \vx^*\|^2 + \|\vy_0 - \vy^*\|^2.
\end{equation}
\fi
In both cases, the growth of $A_k$ can be bounded below as
\ificml
\begin{align*}
    A_k \geq &\frac{1}{\sqrt{2}R} \max \Big\{{k},\; \Big(1 +  \frac{\sqrt{\sigma\gamma}}{\sqrt{2}R}\Big)^{k-1}, \\
    & \frac{\sigma}{9\sqrt{2}R}\Big( [k - k_0]_+ + \max\big\{3.5 \sqrt{R},\, 1\big\}\Big)^{2},\\  %
     &\frac{\gamma}{9\sqrt{2}R}\Big( [k - k_0']_+ + \max\big\{3.5 \sqrt{R},\, 1\big\}\Big)^{2} %
    \Big\},
\end{align*}
\else
\begin{align*}
    A_k \geq \frac{1}{\sqrt{2}R} \max \Big\{&{k},\; \Big(1 +  \frac{\sqrt{\sigma\gamma}}{\sqrt{2}R}\Big)^{k-1}, \\
    & \frac{\sigma}{9\sqrt{2}R}\Big( [k - k_0]_+ + \max\big\{3.5 \sqrt{R},\, 1\big\}\Big)^{2},\\  %
     &\frac{\gamma}{9\sqrt{2}R}\Big( [k - k_0']_+ + \max\big\{3.5 \sqrt{R},\, 1\big\}\Big)^{2} %
    \Big\},
\end{align*}
\fi
where $[\cdot]_+ = \max\{\cdot, 0\}$, $k_0 = \lceil \frac{\sigma}{9\sqrt{2}R} \rceil,$ and $k_0' = \lceil \frac{\gamma}{9\sqrt{2}R} \rceil.$
\end{restatable}

\begin{remark}\label{rem:pda2-convergence-of-the-iterates}
As $\sigma \geq 0$ and $\gamma \geq 0,$ Theorem~\ref{thm:general} guarantees that all iterates of \pda~remain within a bounded set, due to Eq.~\eqref{eq:pda2-bnded-diam}. In particular, $\vx_k \in \gB(\vx^*, r_0),$ $\vy_k \in \gB(\vy^*, \sqrt{2}r_0),$ where $r_0 = \sqrt{\|\vx_0 - \vx^*\|^2 + \|\vy_0 - \vy^*\|^2}$ and $\gB(\vz, r)$ denotes the Euclidean ball of radius $r$,
 centered at $\vz.$ 
 Moreover, by rearranging Eq.~\eqref{eq:pda2-bnded-diam}, we can conclude that  $\|\vx^* - \vx_k\|^2 \leq \frac{{r_0}^2}{1 + \sigma A_k}$ and ${\|\vy^* - \vy_k\|^2} \leq \frac{2 {r_0}^2}{1 + \gamma A_k}.$ 
 \end{remark}

\begin{remark}\label{rem:thm-general-convergence}
Observe that when the domains of $g^*$ and $\ell$ are bounded (i.e., when $D_{\gX} <\infty$, $D_{\gY}< \infty,$ and, in particular, in the setting of constrained optimization over compact sets), %
Theorem~\ref{thm:general} implies the following bound on the  primal-dual gap $\Gap(\vxt_k, \vyt_k) \leq \frac{{D_{\gX}}^2 + {D_{\gY}}^2}{2 A_k}$. This bound can be shown to be optimal, using results from~\citet{ouyang2018lower}. 
For unbounded domains of $g^*$ and $\ell$, it is generally not possible to have any finite bound on $\Gap(\vx, \vy)$ unless $(\vx, \vy) = (\vx^*, \vy^*)$ (for a concrete example, see, e.g.,~\citet{diakonikolas2020halpern}). 
In such a case, it is common to restrict $\vu, \vv$ to bounded sets that include $\vx^*, \vy^*$, such as $\gB(\vx^*, r_0),$ $\gB(\vy^*, \sqrt{2}r_0)$ from Remark~\ref{rem:pda2-convergence-of-the-iterates} \citep{chambolle2011first}. 
\end{remark}

\begin{remark}\label{rem:thm-general-fun-val-gap}
To bound the function value gap $f(\vxt_k) - f(\vx^*)$ for Problem~\eqref{eq:prob} using Theorem~\ref{thm:general}, we need only that $D_{\gY}$ is bounded, leading to the bound $f(\vxt_k) - f(\vx^*)\leq \frac{4{r_0}^2 + {D_{\gY}}^2}{A_k},$ 
where $r_0 = \sqrt{\|\vx_0 - \vx^*\|^2 + \|\vy_0 - \vy^*\|^2}$  as in Remark~\ref{rem:pda2-convergence-of-the-iterates}, since for $\vu \in \gB(\vx^*, r_0)$ we have that $\|\vu - \vx_0\| \leq 2r_0$. 
To see this, note that, as the iterates $\vx_i$ of \pda~are guaranteed to remain in $\gB(\vx^*, r_0)$ (by Remark~\ref{rem:pda2-convergence-of-the-iterates}), there is no difference between applying this algorithm to $f$ or to $f + I_{\gB(\vx^*, r_0)},$ where $I_{\gB(\vx^*, r_0)}$ is the indicator function of $\gB(\vx^*, r_0).$ This allows us to restrict $\vu \in \gB(\vx^*, r_0)$ when bounding $f(\vxt_k) - f(\vx^*)$ by $\Gap^{\vu, \vv}(\vxt_k, \vyt_k).$ %
Note that for typical instances of nonsmooth ERM problems, the domain $\gY$ of $g^*$ is compact. Further, if $g^*$ is strongly convex ($\gamma > 0$), then the set $\tilde{\gY} := \{\argmax_{\vy \in \sR^n} \innp{\mB\vx, \vy} - g^*(\vy): \vx \in \gB(\vx^*, r_0)\}$ is guaranteed to be compact. This claim follows from standard results, as in this case $\argmax_{\vy \in \sR^n} \innp{\mB\vx, \vy} - g^*(\vy) = \nabla g(\mB\vx)$  (by the standard Fenchel-Young inequality; see, e.g.,~\citet[Proposition 11.3]{rockafellar2009variational}) and $g$ is $\frac{1}{\gamma}$-smooth. Thus, $\sup_{\vv, \vy  \in \tilde{\gY}}\|\vv - \vy\| = \sup_{\vx, \vu \in \gB(\vx^*, r_0)}\|\nabla g(\mB\vx) - \nabla g(\mB \vu)\| \leq \frac{R}{\gamma}r_0.$  %
\end{remark}

\section{Variance Reduction via Primal-Dual Accelerated Dual Averaging }\label{sec:vrpdada}

\begin{algorithm*}[t!]
\caption{Variance Reduction via Primal-Dual Accelerated Dual Averaging (\vrpda~)}\label{alg:vrpda}
\begin{algorithmic}[1]
\STATE \textbf{Input: } $(\vx_0, \vy_0)\in\gX\times \gY, (\vu, \vv)\in \gX \times \gY$, $\sigma \ge 0, R' >0, K, n.$
\STATE $\phi_0(\cdot) = \frac{1}{2}\|\cdot - \vx_0\|^2, \psi_{0}(\cdot) = \frac{1}{2}\|\cdot - \vy_0\|^2$.
\STATE $a_0 = A_0 = 0$, $\tilde{a}_1 = \frac{1}{2R'}$.
\STATE $\vy_1 = \argmin_{\vy\in\sR^n}\{ \tilde{\psi}_1(\vy) :=  \psi_{0}(\vy) + \tilde{a}_1 (\langle - \mB \vx_0,\vy - \vv\rangle +  g^*(\vy))\}$.
\STATE $\vz_1 = \mB^T\vy_1$.
\STATE $\vx_1 =  \argmin_{\vx\in\sR^d}\{\tilde{\phi}_1(\vx) :=  \phi_{0}(\vx) + \tilde{a}_1 (\langle {\vx} - \vu, \vz_1\rangle+\ell(\vx))\}$.
\STATE $\psi_1 := n\tilde{\psi}_1, \phi_1:= n\tilde{\phi}_1, a_1 = A_1 = n \tilde{a}_1, a_2 = \frac{1}{n-1}a_1, A_2 = A_1 + a_2$.
\FOR{$k = 2,3,\ldots, K$}
\STATE $\bar{\vx}_{k-1} =  \vx_{k-1} + \frac{a_{k-1}}{a_{k}}(\vx_{k-1} - \vx_{k-2}).$
\STATE Pick $j_k$ uniformly at random in $[n].$
\STATE $\vy_{k} = \argmin_{\vy\in\sR^n}\{  \psi_k(\vy) =  \psi_{k-1}(\vy) + a_k (- \vb_{j_k}^T \bar{\vx}_{k-1} (y_{j_k} - v_{j_k}) +  g_{j_k}^*(y_{j_k} ))\}$. \label{ln:vrpda-y_k-comp}
\STATE $\vx_{k} = \argmin_{\vx\in\sR^d}\{  \phi_k(\vx) =  \phi_{k-1}(\vx) + a_k (\langle {\vx} - \vu,  \vz_{k-1} + (y_{k, j_k} - y_{k-1, j_k})\vb_{j_k}  \rangle+\ell(\vx))\}.$ \label{ln:vrpda-x_k-comp}
\STATE $\vz_k = \vz_{k-1} + \frac{1}{n}(y_{k, j_k} - y_{k-1, j_k})\vb_{j_k}.$
\STATE  $a_{k+1} = \min \Big( \big(1+\frac{1}{n-1} \big) a_{k}, \frac{\sqrt{n(n+\sigma A_{k})}}{2R'} \Big)$, $A_{k+1} = A_k + a_{k+1}.$
\ENDFOR
\STATE \textbf{return } $\vx_K$ or $\tilde{\vx}_K := \frac{1}{A_K}\sum_{i=1}^K a_i \vx_i$.
\end{algorithmic}%
\end{algorithm*}

We now study the finite-sum form \eqref{eq:primal-dual-finite} of \eqref{eq:primal-dual-deter}, making use of the properties of the finite-sum terms described in Assumption~\ref{ass:vr}.
In Algorithm \ref{alg:vrpda}, we describe %
\vrpda\, which is a randomized coordinate variant of the \pda~algorithm from Section~\ref{sec:pdada}. 
By extending the unified nature of \pda, \vrpda~ provides a unified and simplified treatment for both the general convex-general concave $(\sigma=0)$ setting and the strongly convex-general concave setting.

To provide an algorithm with complexity better than the deterministic counterpart \pda, we combine the deterministic initialization strategy of full primal-dual update in Steps 4-6 with randomized primal-dual updates in the main loop---a strategy inspired by the recent paper of \citet{song2020variance}. 
The use of the factor $n$ during initialization, in Step 7, helps to cancel an error term of order $O(n)$ in the analysis.

The main loop (Steps 8-15) randomizes the main loop of \pda~ by introducing sampling in Step 10 and adding an auxiliary variable $\vz_k$ that is updated with $O(d)$ cost in Step 13. ($\vz_1$ is initialized in Step 5). 
In Step 11, we update the estimate sequence $\psi_k$ by adding a term involving only the $j_k$ component of the finite sum, rather than the entire sum, as is required in Step 8 of Algorithm \ref{alg:pda}. 
As a result, although we define the estimate sequence for the entire vector $\vy_k$, each update to $\vy_k$ requires updating {\em only the $j_k$ coordinate} of $\vy_k$. 
In Step 12, we use a ``variance reduced gradient'' $\vz_{k-1} + (y_{k, j_k} - y_{k-1, j_k})\vb_{j_k}$ to update  $\phi_k$, helping to cancel the error from the randomized update of Step 11. 
The update of the sequences $\{a_{k}\}$, $\{A_{k}\}$ appears at the end of the main loop,  to accommodate their modified definitions. %
The modified update for  $a_{k+1}$ ensures that $a_k$ cannot have exponential growth with a rate  higher than $\big(1+\frac{1}{n-1}\big)$, which is an intrinsic constraint for sampling with replacement (see %
\citet{song2020variance,hannah2018breaking}).

Finally, as Algorithm \ref{alg:vrpda} is tailored to the nonsmooth ERM problem \eqref{eq:prob}, we only return the last iterate $\vx_k$ or the weighed average iterate $\tilde{\vx}_k$ on the primal side, %
even though we provide guarantees for both primal and dual variables.

Algorithm~\ref{alg:vrpda} provides sufficient detail for %
the convergence analysis, but its efficient implementation is not immediately clear, due especially to Step 11. 
An implementable version is described in Appendix~\ref{appx:comp-considerations}, showing that the per-iteration cost is $O(d)$ and that $O(n)$ additional storage is required.

Our main technical result is summarized in Theorem~\ref{thm:vr}. Its proof relies on three main technical lemmas that bound the growth of estimate sequences $\phi_k(\vx_k)$ and $\psi_k(\vy_k)$ below and above. Proofs are provided in %
Appendices~\ref{appx:omitted-proofs-vr} and \ref{appx:growth-of-seqs}.

\begin{restatable}{theorem}{mainthmvr}\label{thm:vr}
Suppose that Assumption \ref{ass:vr} holds.
Then for any  $(\vu, \vv) \in \gX \times \gY,$ the vectors $\vx_k$, $\vy_k$, $k=2,3,\dotsc,K$ and the average $\tilde{\vx}_K$  generated by Algorithm~\ref{alg:vrpda}  satisfy the following bound for $k =2,3,\dotsc,K$:
\begin{equation}\notag
    \E[\Gap^{\vu, \vv}(\vxt_k, \vyt_k)] \leq \frac{n(\|\vu - \vx_0\|^2 + \|\vv - \vy_0\|^2)}{2A_k},
\end{equation}
where 
$
\tilde{\vy}_K  := \frac{na_K\vy_K+  \sum_{i=2}^{K-1} ( n a_i - (n-1)a_{i+1})\vy_i}{A_K}. 
$ 
Moreover, if $(\vx^*, \vy^*)$ is a primal-dual solution to~\eqref{eq:primal-dual-deter}, then 
\ificml
\begin{equation*}
\begin{aligned}
    &\E\Big[\frac{n}{4}\|\vy^* - \vy_k\|^2 + \frac{n + \sigma A_k}{2}\|\vx^* - \vx_k\|^2 \Big]\\
    &\hspace{1in}\leq \frac{n(\|\vx^* - \vx_0\|^2 + \|\vy^* - \vy_0\|^2)}{2}.
\end{aligned}
\end{equation*}
\else
\begin{equation*}
    \E\Big[\frac{n}{4}\|\vy^* - \vy_k\|^2 + \frac{n + \sigma A_k}{2}\|\vx^* - \vx_k\|^2 \Big] \leq \frac{n(\|\vx^* - \vx_0\|^2 + \|\vy^* - \vy_0\|^2)}{2}.
\end{equation*}
\fi
In both cases,  $A_k$ is bounded below as follows:
\ificml
\begin{align*}
    A_k \geq \max \Big\{&\frac{n-1}{2R'}\Big(1 + \frac{1}{n-1}\Big)^k \mathds{1}_{k \leq k_0},\\
    & \frac{(n-1)^2 \sigma}{(4R')^2 n}(k-k_0 + n-1)^2 \mathds{1}_{k \geq k_0},\\
    & \frac{n(k - K_0 + n - 1)}{2R'}\mathds{1}_{k \geq K_0}\Big\},
\end{align*}
\else
\begin{align*}
    A_k \geq \max \Big\{&\frac{n-1}{2R'}\Big(1 + \frac{1}{n-1}\Big)^k \mathds{1}_{k \leq k_0},\;
     \frac{(n-1)^2 \sigma}{(4R')^2 n}(k-k_0 + n-1)^2 \mathds{1}_{k \geq k_0},\\
    & \frac{n(k - K_0 + n - 1)}{2R'}\mathds{1}_{k \geq K_0}\Big\},
\end{align*}
\fi
where $\mathds{1}$ denotes the indicator function, $K_0 = \big\lceil \frac{\log(n)}{\log(n) - \log(n-1)}\big\rceil,$ $k_0 = \big\lceil \frac{\log B_{n,\sigma,R'}}{\log (n) - \log (n-1)} \big\rceil,$ and
\begin{align*}
B_{n,\sigma,R'} & = 
 \frac{\sigma n (n-1)}{4 R'} +  \sqrt{\Big(  \frac{\sigma n (n-1)}{4 R'}\Big)^2 + n^2} \\
& \geq n\max\Big\{1, \frac{\sigma (n-1)}{2 R'}\Big\}.
\end{align*}
\end{restatable}

{Observe that, due to the randomized nature of the algorithm, the convergence bounds are obtained in expectation w.r.t.~the random choices of coordinates $j_k$ over iterations. Now let us comment on the iteration complexity of \vrpda, given target error $\epsilon > 0.$ For concreteness, let $D^2 := \|\vu - \vx_0\|^2 + \|\vv- \vy_0\|^2,$ where $D^2$ can be bounded using the same reasoning as in Remarks~\ref{rem:thm-general-convergence} and \ref{rem:thm-general-fun-val-gap}. To bound the gap by $\epsilon,$ we need $A_k \geq \frac{n D^2}{2\epsilon}$. When $\epsilon \geq \frac{n R' D^2}{(n-1)B_{n, R', \sigma}}$, then $k = \Big\lceil \frac{\log(\frac{nR'D^2}{(n-1)\epsilon})}{\log(n) - \log(n-1)} \Big\rceil = O \big(n \log(\frac{R'D}{\epsilon})\big)$ iterations suffice, as in this case $k \leq k_0$. When $\epsilon < \frac{n R' D^2}{(n-1)B_{n, \sigma,R'}}$, then the bound on $k$ is obtained by ensuring that either of the last two terms bounding $A_k$ below in Theorem~\ref{thm:vr} is bounded below by $\frac{n D^2}{2\epsilon},$ leading to $k = O\big(n\log(B_{n, \sigma,R'}) + \min\{\frac{R'D}{\sqrt{\sigma \epsilon}},\, \frac{R'D^2}{\epsilon}\}\big).$}

\begin{remark}\label{rem:feature-selection}
{%
The usefulness of our results extends beyond ERM problems~\eqref{eq:prob}, due to the symmetry of the primal-dual problem~\eqref{eq:primal-dual-deter} that \vrpda~solves. In particular, the symmetry allows exchanging roles of $\vx$ and $\vy$ in cases where $d \gg n$, $\ell$ is decomposable over the coordinates with each coordinate function having an efficiently computable prox-oracle, and $g^*$ has an efficiently computable prox-oracle. A concrete example are feature selection problems, including, e.g., Lasso~\cite{tibshirani1996regression}, elastic net~\cite{zou2005regularization}, square-root Lasso/distributionally robust optimization~\cite{belloni2011square,blanchet2019robust} and $\ell_1$-regularized logistic regression~\cite{ravikumar2010high}. For all these problems (ignoring the dependence on $R', D$), the overall complexity that can be achieved with \vrpda~is $O(nd \log(\min\{d, 1/\epsilon\}) + \min\{\frac{n}{\epsilon}, \frac{n}{\sqrt{\sigma \epsilon}}\})$; i.e., we can reduce the dependence on $d$ (by a factor $d$ when $\epsilon$ is small enough). To the best of our knowledge, such a result cannot be obtained with any other variance reduction methods.}
\end{remark}

\vspace{-5mm}

\section{Numerical Experiments}\label{sec:num-exp}
\vspace{-2mm}

\vspace{-3mm}
\begin{figure*}[ht!]
\centering
\subfigure[\texttt{a9a}, $\sigma = 0,$ average]{\includegraphics[width=0.24\textwidth]{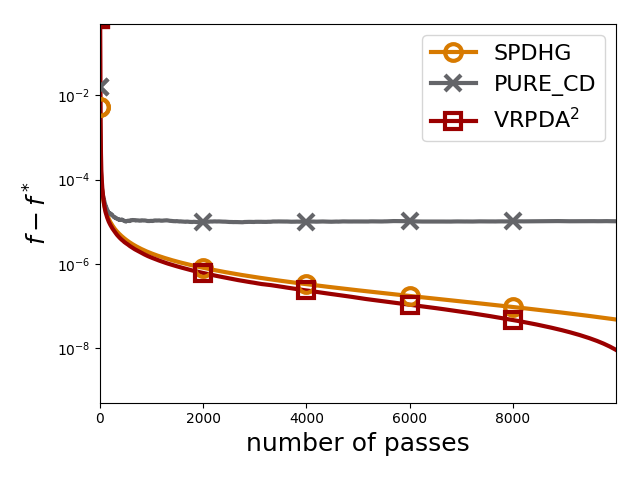}\label{fig:smooth-5-blog}}
\subfigure[\texttt{a9a}, $\sigma = 0,$ last]{\includegraphics[width=0.24\textwidth]{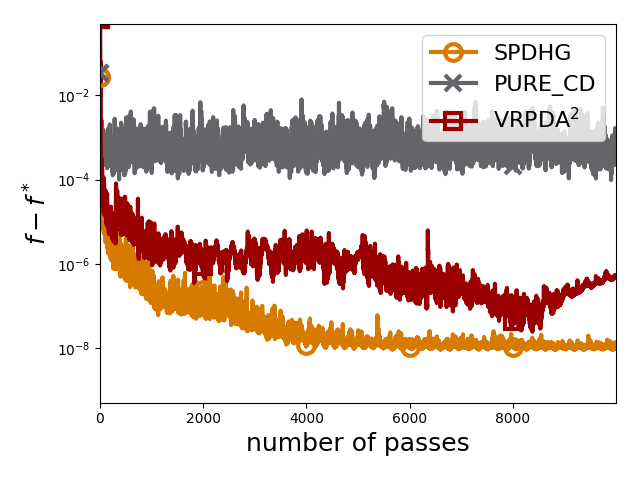}\label{fig:smooth-10-blog}}
 \subfigure[\texttt{MNIST}, $\sigma = 0,$ average]{\includegraphics[width=0.24\textwidth]{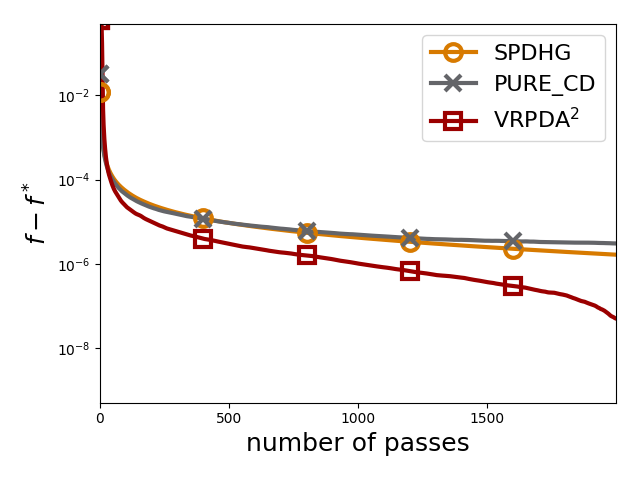}\label{fig:lip-med-20}}
 \subfigure[\texttt{MNIST}, $\sigma = 0,$ last]{\includegraphics[width=0.24\textwidth]{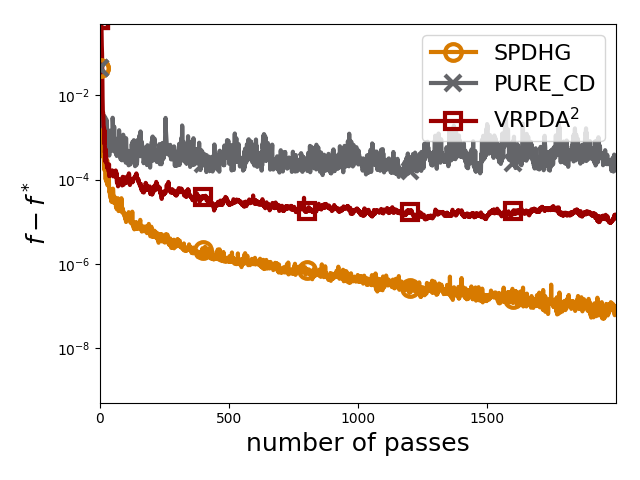}\label{fig:lip-med-40}}
\subfigure[\texttt{a9a}, $\sigma = 10^{-8},$  average]{\includegraphics[width=0.24\textwidth]{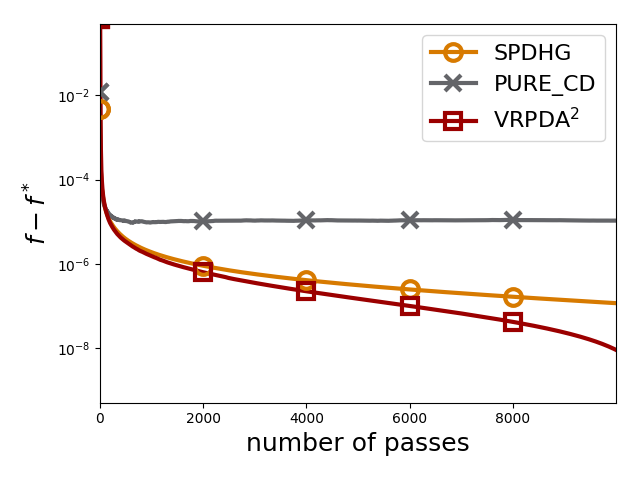}\label{fig:smooth-20-blog}}
\subfigure[\texttt{a9a}, $\sigma = 10^{-8},$  last]{\includegraphics[width=0.24\textwidth]{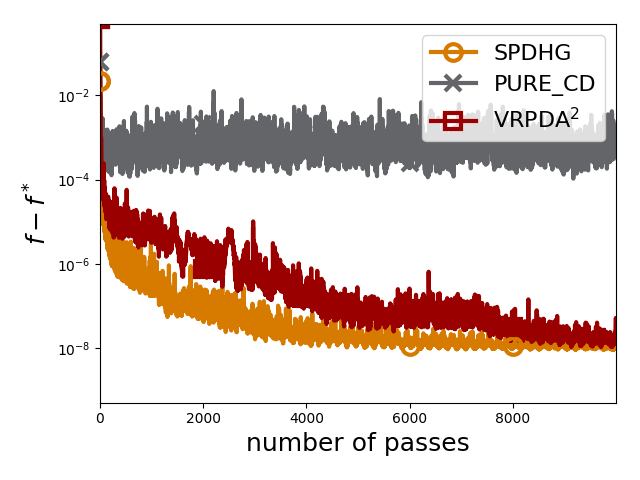}\label{fig:smooth-40-blog}}
\subfigure[\texttt{MNIST}, $\sigma = 10^{-8},$ average]{\includegraphics[width=0.24\textwidth]{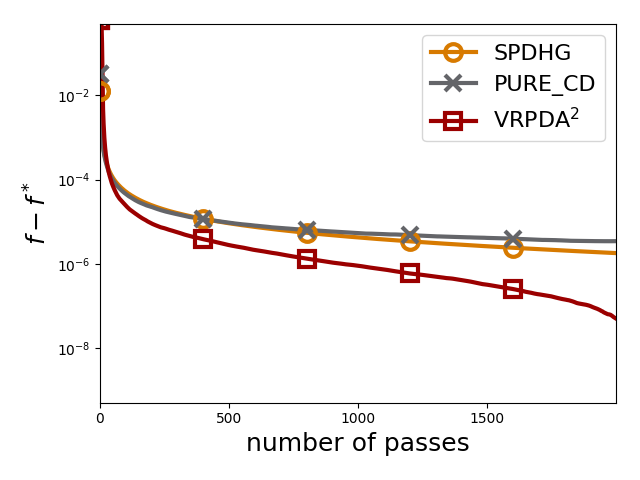}\label{fig:acc-lip-5-blog}}
 \subfigure[\texttt{MNIST}, $\sigma = 10^{-8},$ last]{\includegraphics[width=0.24\textwidth]{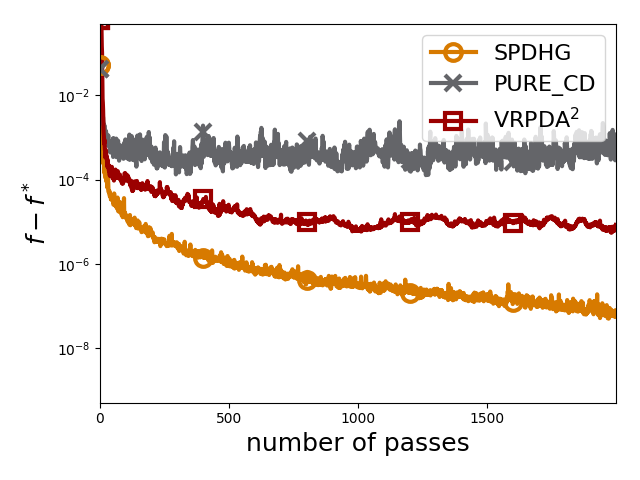}\label{fig:acc-lip-10-blog}}
\subfigure[\texttt{a9a}, $\sigma = 10^{-4},$  average]{\includegraphics[width=0.24\textwidth]{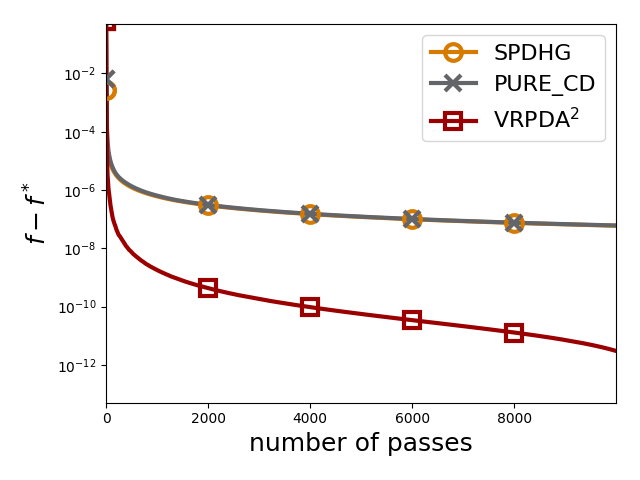}\label{fig:lip-5-blog}}
 \subfigure[\texttt{a9a}, $\sigma = 10^{-4},$  last]{\includegraphics[width=0.24\textwidth]{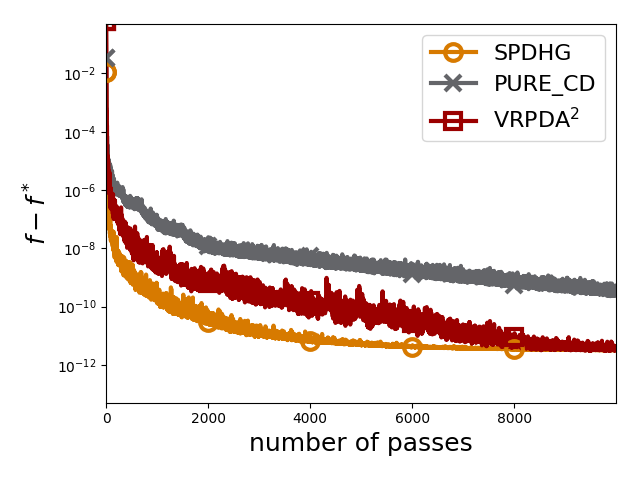}\label{fig:lip-10-blog}}
 \subfigure[\texttt{MNIST}, $\sigma = 10^{-4},$ average]{\includegraphics[width=0.24\textwidth]{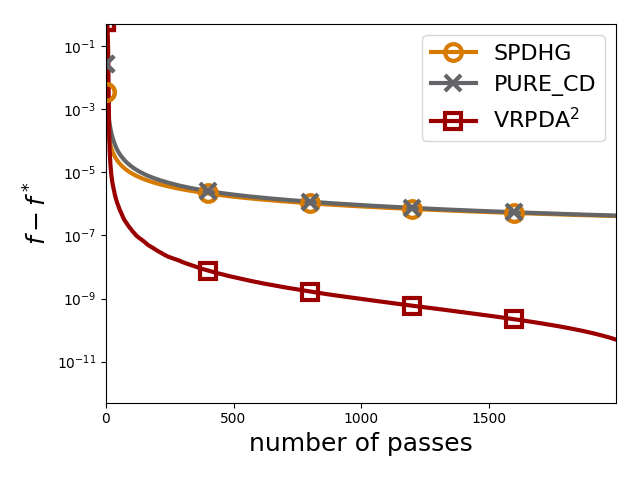}\label{fig:acc-lip-med-20}}
 \subfigure[\texttt{MNIST}, $\sigma = 10^{-4},$ last]{\includegraphics[width=0.24\textwidth]{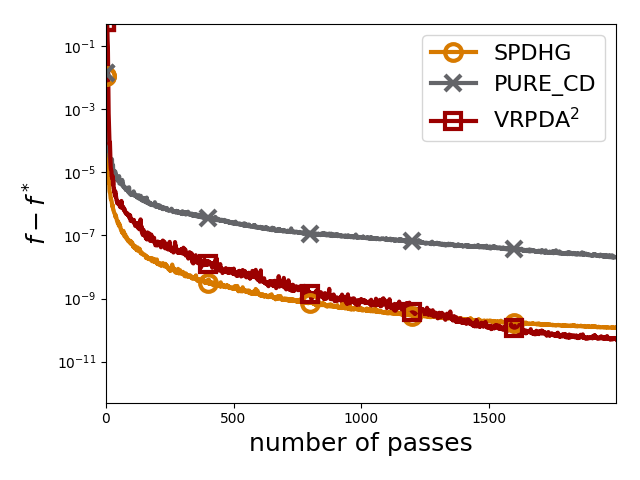}\label{fig:acc-lip-med-40}}
\caption{Comparison of \vrpda~to \textsc{spdhg} and \textsc{pure\_cd} run for the elastic net-regularized SVM, on \texttt{a9a} and \texttt{MNIST} datasets. In all the plots, $\sigma$ is the strong convexity parameter of the regularizer $\ell$; ``last'' refers to the last iterate, ``average'' to the average iterate. For all problem instances, \vrpda~attains either similar or improved convergence compared to other algorithms.}
\label{fig:vrpda}
\end{figure*}

We study the performance of \vrpda~using
the elastic net-regularized support vector machine (SVM) problem, which corresponds to  \eqref{eq:prob} with $g_i(\vb_i^T\vx) = \max\{1-c_i\vb_i^T\vx, 0\},$ $c_i\in\{1, -1\}$ and $\ell(\vx) = \lambda\|\vx\|_1 + \frac{\sigma}{2}\|\vx\|_2^2$, $\lambda\ge 0,$ $\sigma \ge 0$. This problem is nonsmooth and general convex if $\sigma=0$ or strongly convex if $\sigma>0.$ Its primal-dual formulation is 
\begin{equation} 
\begin{gathered}
\min_{\vx\in\sR^d}\max_{{-1\le y_i \le 0,\, i\in[n]}}L(\vx, \vy),\\
L(\vx, \vy) = \frac{1}{n}\sum_{i=1}^n y_i\left(  \langle c_i\vb_i,  \vx\rangle - 1\right) + \lambda\|\vx\|_1 + \frac{\sigma}{2}\|\vx\|_2^2. 
\end{gathered}\nonumber
\end{equation}
We compare \vrpda~with two competitive algorithms \textsc{pdhg} \cite{chambolle2018stochastic} and \textsc{pure\_cd} \cite{alacaoglu2020random} on standard \texttt{a9a} and \texttt{MNIST} datasets from the LIBSVM library~\cite{LIBSVM}.%
\footnote{For each sample of \texttt{MNIST}, we reassign the label as $1$ if it is in $\{5,6,\ldots,9\}$ and $-1$ otherwise.} %
Both datasets \texttt{a9a} and \texttt{MNIST} we use are large scale with $n = 32561$, $d=123$ for \texttt{a9a}, and $n = 60000, d = 780$ for \texttt{MNIST}. 
For simplicity, we normalize each data sample to unit Euclidean norm, thus the Lipschitz constants (such as $R'$ in \vrpda) in the algorithms are at most $1$. Then we tune the Lipschitz constants in $\{0.1, 0.25, 0.5, 0.75, 1\}$\footnote{In our experiments, all the algorithms  diverge when the Lipschitz constant is set to $0.1$.}. As is standard for ERM, we plot the function value gap of the primal problem \eqref{eq:prob} in terms of the number of passes over the dataset. The plotted function value gap was evaluated using an estimated value $\tilde{f}^*$ of $f^*  = \argmin_{\vx} f(\vx).$ For the plots to depict an accurate estimate of the function value gap, it is required that the true function value gap $f - f^*$ dominates the error of the estimate $\tilde{f}^* - f^*.$ In our numerical experiments, this is achieved by running the algorithms for $30\times$ iterations compared to what is shown in the plots, choosing the lowest function value $f_{\min}$ seen over the algorithm run and over all the algorithms, and then setting $\tilde{f}^* = f_{\min} - \delta,$ where $\delta$ is chosen either as $10^{-8}$ or $10^{-13}$, depending on the value of $\sigma.$

In the experiments, we fix the $\ell_1$-regularization parameter $\lambda=10^{-4}$ and vary $\sigma\in\{0, 10^{-8}, 10^{-4}\}$ to represent the general convex, ill-conditioned strongly convex, and well-conditioned strongly convex settings, respectively. For all the settings, we provide the comparison in terms of the average and last iterate\footnote{\textsc{spdhg} and \textsc{pure\_cd} provide no results for the average iterate in the nonsmooth and strongly convex setting, so we use simple uniform average for both of them.}. As can be observed from Fig.~\ref{fig:vrpda}, the average iterate yields much smoother curves, decreasing monotonically, and is generally more accurate than the last iterate. 
This is expected for the nonsmooth and general convex setting, as there are no theoretical guarantees for the last iterate, while for other cases the guarantee for the last iterate is on the distance to optimum, not the primal gap. 
As can be seen in Fig.~\ref{fig:vrpda}, the average iterate of \vrpda~is either competitive with or improves upon \textsc{spdhg} and \textsc{pure\_cd}.

\begin{figure*}[ht!]
\centering
\subfigure[\texttt{a9a}, $\sigma = 0,$ average]{\includegraphics[width=0.24\textwidth]{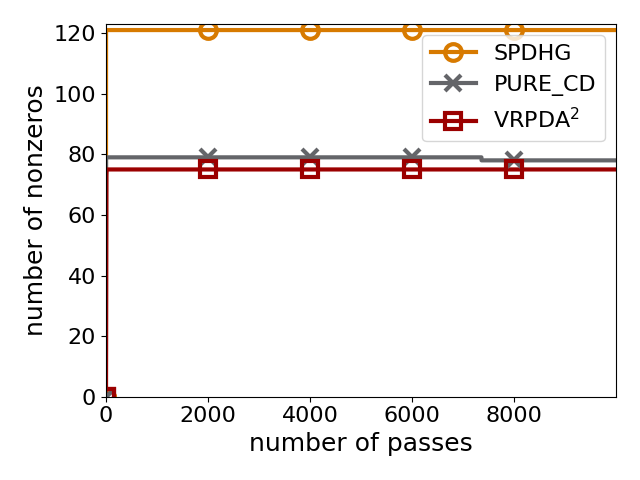}}
\subfigure[\texttt{a9a}, $\sigma = 0,$ last]{\includegraphics[width=0.24\textwidth]{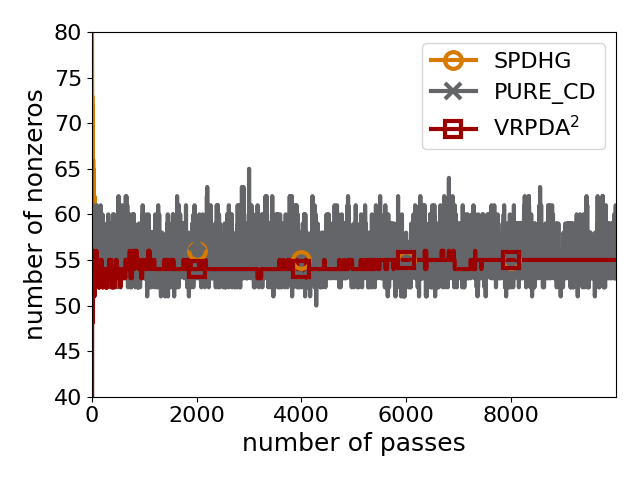}}
 \subfigure[\texttt{MNIST}, $\sigma = 0,$ average]{\includegraphics[width=0.24\textwidth]{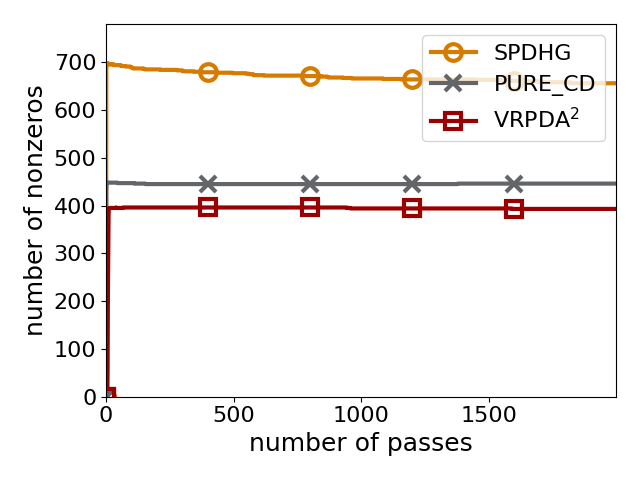}}
 \subfigure[\texttt{MNIST}, $\sigma = 0,$ last]{\includegraphics[width=0.24\textwidth]{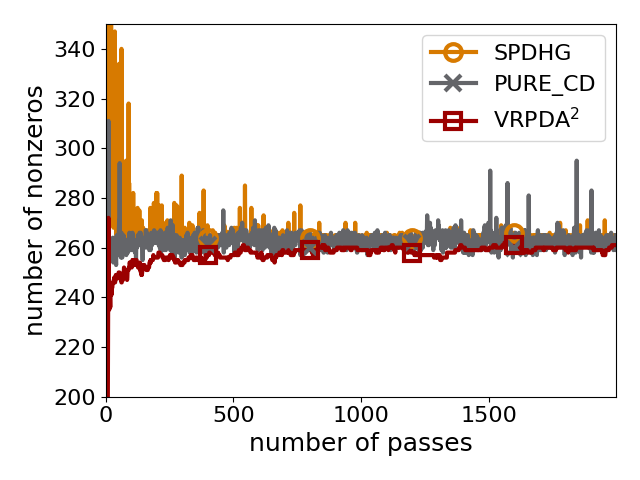}}
\subfigure[\texttt{a9a}, $\sigma = 10^{-8},$  average]{\includegraphics[width=0.24\textwidth]{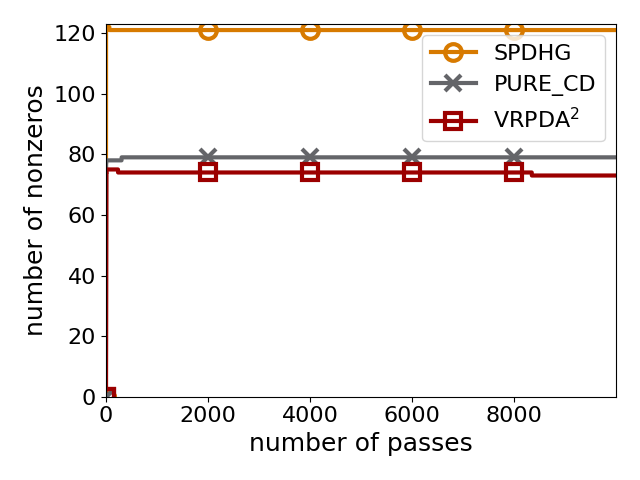}}
\subfigure[\texttt{a9a}, $\sigma = 10^{-8},$  last]{\includegraphics[width=0.24\textwidth]{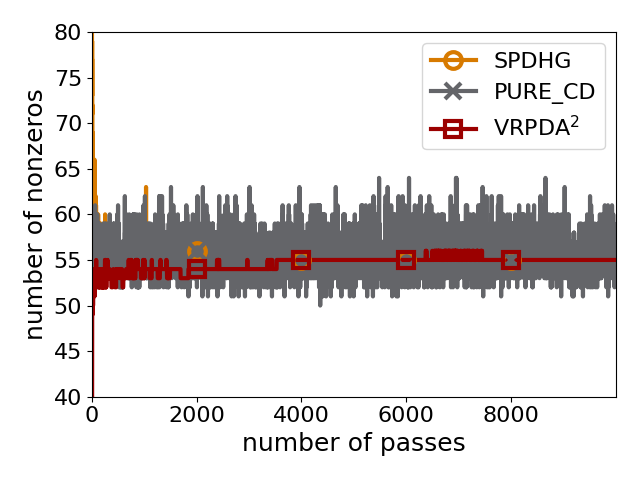}}
\subfigure[\texttt{MNIST}, $\sigma = 10^{-8},$ average]{\includegraphics[width=0.24\textwidth]{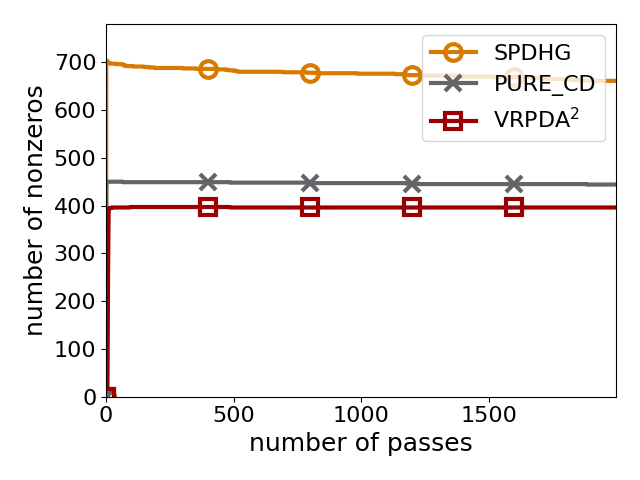}}
 \subfigure[\texttt{MNIST}, $\sigma = 10^{-8},$ last]{\includegraphics[width=0.24\textwidth]{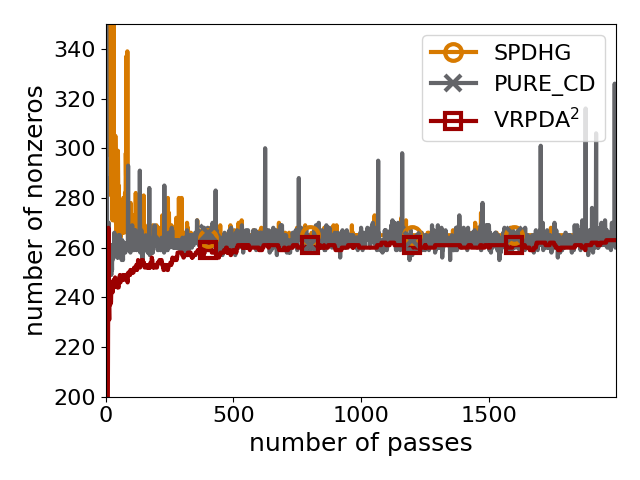}}
\subfigure[\texttt{a9a}, $\sigma = 10^{-4},$  average]{\includegraphics[width=0.24\textwidth]{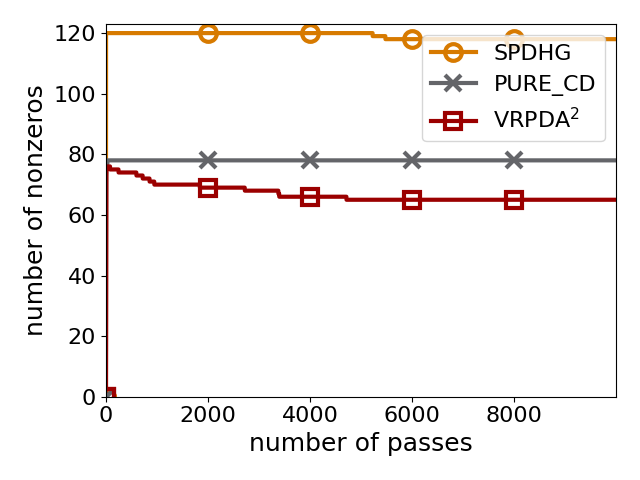}}
 \subfigure[\texttt{a9a}, $\sigma = 10^{-4},$  last]{\includegraphics[width=0.24\textwidth]{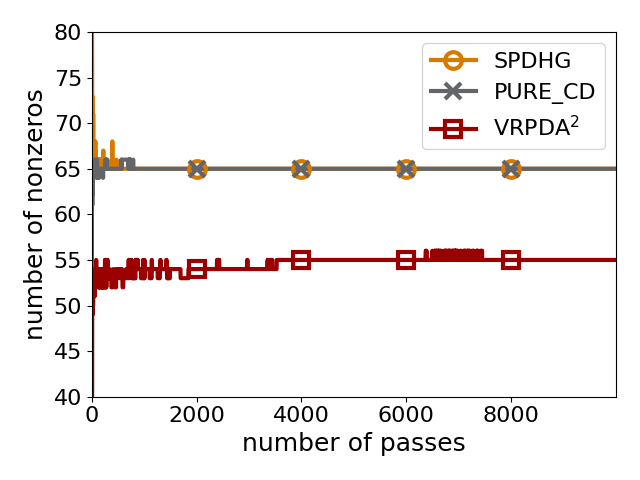}}
 \subfigure[\texttt{MNIST}, $\sigma = 10^{-4},$ average]{\includegraphics[width=0.24\textwidth]{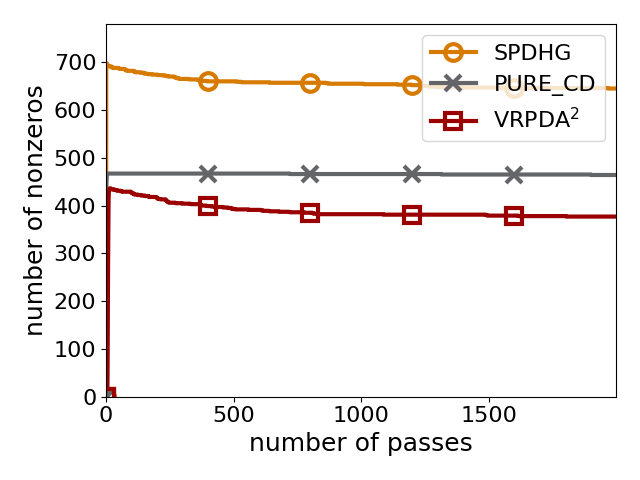}}
 \subfigure[\texttt{MNIST}, $\sigma = 10^{-4},$ last]{\includegraphics[width=0.24\textwidth]{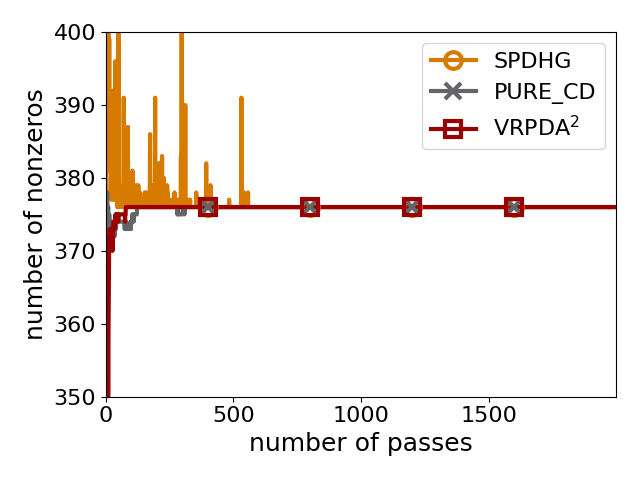}}
\caption{
Comparison of sparsity for \vrpda, \textsc{spdhg}, and \textsc{pure\_cd} run for the elastic net-regularized SVM problem, on \texttt{a9a} and \texttt{MNIST} datasets. In all the plots, $\sigma$ is the strong convexity parameter of the regularizer $\ell$; ``last'' refers to the last iterate, ``average'' to the average iterate. For all problem instances, \vrpda~generally constructs the sparsest solutions out of the three algorithms. (The number of nonzeros is computed by counting the elements with absolute value larger than $10^{-7}$.)
}
\label{fig:vrpda-sparsity}
\end{figure*}

As can be observed from Figure \ref{fig:vrpda}, there is a noticeable difference in the performance of all the algorithms when their function value gap is evaluated at the average iterate versus the last iterate. For \vrpda~that is a dual averaging-style method with the role of promoting sparsity \cite{xiao2010dual}, this difference comes from the significantly different sparsity of the average iterate and last iterate: as shown in Figure \ref{fig:vrpda-sparsity}, average iterate is less sparse but provides more accurate fit, while the last iterate is sparser (and thus more robust) but less accurate. For \spdhg, the last iterate is significantly more accurate than the average iterate in the strongly convex settings (i.e., for $\sigma\in\{10^{-8}, 10^{-4}\}$), because simple uniform average we use may not be the best choice for the two settings.\footnote{However, as we mentioned in the main body, \spdhg~\citep{chambolle2018stochastic} provides no results for the average iterate in the strongly convex setting.} Meanwhile, compared with \vrpda, the better performance of \spdhg~ in terms of the last iterate is partly due to the fact that it is a mirror descent-style algorithm with less sparse last iterate. In our experiments, the \pure~algorithm is always worse than \vrpda~and \spdhg, which is partly consistent with its worse convergence guarantee as shown in Table \ref{tb:result}. However, as \pure~is particularly designed for sparse datasets, it may have a better runtime performance  for sparse datasets (as shown in \citet{alacaoglu2020random}), which is beyond the scope of this paper.

Meanwhile,  the performance of the average iterate of \vrpda~and the last iterate of \spdhg~is almost the same (both figures for the average iterate and the last iterate under the same setting use the same scale), which is surprising as \vrpda~has $n$-times better theoretical guarantees than \spdhg~for small $\epsilon$. The better theoretical guarantees of \vrpda~comes from the particular  initialization strategy inspired by \citet{song2020variance}. However, similar to the experimental results in  \citet{song2020variance}, we do not observe the performance gain of the initialization strategy in practice (despite the fact that it does not worsen the empirical performance). Thus, it is of interest to explore whether the initialization strategy is essential for the improved algorithm performance or %
if it is only needed for the theoretical argument to go through.

Finally, as most papers do, we chose not to plot $\|\vx_k - \vx^*\|$ as it would require estimating $\vx^*,$ which is much harder than estimating $f^*.$ In particular, estimating $\vx^*$ is out of reach for problems that are not strongly convex even if $f$ were smooth, due to known lower bounds (see, e.g., \citet[Theorem~2.1.7]{nesterov2018lectures}). For strongly convex problems, we can obtain an estimate $\vxt^*$ of $\vx^*$ from the function value gap $f(\vxt^*) - f(\vx^*)$ using strong convexity of $f,$ as in this case $\|\vxt^* - \vx^*\| \leq \sqrt{\frac{2}{\sigma}(f(\vxt^*) - f(\vx^*))}$. However, obtaining error $\delta = \|\vxt^* - \vx^*\|$ from the function value gap would require $f(\vxt^*) - f(\vx^*) \leq \frac{\sigma}{2}\delta^2,$ which even for $\delta = 10^{-8}$ and the ``well-conditioned'' setting of our experiments with $\sigma = 10^{-4}$ would require minimizing $f$ to error $10^{-20}$, which becomes computationally prohibitive.

\vspace{-3mm}

\section{Discussion}\label{sec:discussion}
\vspace{-2mm}
We introduced a novel \vrpda~algorithm for structured nonsmooth ERM problems that are common in machine learning. \vrpda~leverages the separable structure of common ERM problems to attain improved convergence compared to the state-of-the-art algorithms, both in theory and  practice, even improving upon the lower bounds for (general, non-structured) composite optimization. 
It is an open question to obtain tighter lower bounds for the structured ERM setting to which \vrpda~applies, possibly certifying its optimality, at least for %
small target error $\epsilon.$

\bibliography{ref_icml.bib}
\ificml
\balance
\bibliographystyle{icml2021}
\else
\bibliographystyle{plainnat}
\fi
\clearpage
\onecolumn
\appendix

\section{Omitted Proofs from Section~\ref{sec:pdada} }\label{appx:omitted-proofs-general}

We start by proving two auxiliary lemmas that bound the growth of the estimate sequences $\phi_k(\vx_k)$ and $\psi_k(\vy_k)$ above and below, respectively. 

\begin{lemma}\label{lem:psi-phi-upper-general}
In Algorithm \ref{alg:pda}, $\forall (\vu,\vv)\in %
\gX\times \gY$
and $k\ge 1,$ we have 
\begin{eqnarray}
\psi_k(\vy_k)&\le& A_k g^*(\vv)  +  \frac{1}{2}\|\vv - \vy_0\|^2 - \frac{1+ \gamma A_k}{2}\|\vv - \vy_k\|^2 \notag\\
\phi_k(\vx_k)&\le& A_k \ell(\vu) + \frac{1}{2}\|\vu - \vx_0\|^2 - \frac{1+ \sigma A_k}{2}\|\vu - \vx_k\|^2. \notag
\end{eqnarray}
\end{lemma}
\begin{proof}
By the definition of $\psi_k(\vy)$ in Algorithm \ref{alg:pda}, it follows that, $\forall k\ge 1,$ 
\begin{equation*}
\psi_k(\vy) = \sum_{i=1}^k a_i(\langle -\mB \bar{\vx}_{i-1}, \vy - \vv\rangle + g^*(\vy) ) + \frac{1}{2}\|\vy - \vy_0\|^2.   \nonumber
\end{equation*}

As, by definition, $\vy_k = \argmin_{\vy \in \sR^n}\psi_k(\vy)$ (see Algorithm \ref{alg:pda}; observe that, by definition of $\psi_k$, it must be $\vy_k \in \gY$), it follows that there exists $\vg_{g^*}(\vy_k) \in \partial g^*(\vy_k)$ %
such that
\begin{equation}\notag
\sum_{i=1}^k a_i( -\mB \bar{\vx}_{i-1} + \vg_{g^*}(\vy_k))+ (\vy_k - \vy_0) = \vzero.
\end{equation}
Thus, for any $\vv\in \sR^n,$ we have that, $\forall k \ge 1$,%
\begin{align*}
\psi_k(\vy_k) &=  \sum_{i=1}^k a_i(\langle -\mB \bar{\vx}_{i-1}, \vy_k - \vv\rangle + g^*(\vy_k) ) + \frac{1}{2}\|\vy_k - \vy_0\|^2 \\ 
&= \sum_{i=1}^k a_i (\langle \vg_{g^*}(\vy_k), \vv - \vy_k\rangle + g^*(\vy_k)) + \langle \vy_k - \vy_0, \vv-\vy_k\rangle + \frac{1}{2}\|\vy_k - \vy_0\|^2  \\
&= A_k (\langle \vg_{g^*}(\vy_k), \vv - \vy_k\rangle + g^*(\vy_k)) + \langle \vy_k - \vy_0, \vv-\vy_k\rangle + \frac{1}{2}\|\vy_k - \vy_0\|^2.  %
\end{align*}
As $g^*$ is assumed to be $\gamma$-strongly convex, we have that %
$\langle \vg_{g^*}(\vy_k), \vv - \vy_k\rangle + g^*(\vy_k) \leq g(\vv) - \frac{\gamma}{2}\|\vv - \vy_k\|^2.$ Thus, using that $\langle \vy_k - \vy_0, \vv-\vy_k\rangle + \frac{1}{2}\|\vy_k - \vy_0\|^2 = \frac{1}{2}\|\vv - \vy_0\|^2 - \frac{1}{2}\|\vv - \vy_k\|^2,$ we further have
\begin{align*}
\psi_k(\vy_k) &\le A_k\Big(g^*(\vv) - \frac{\gamma}{2}\|\vv - \vy_k\|^2\Big) + \frac{1}{2}\|\vv - \vy_0\|^2 - \frac{1}{2}\|\vv - \vy_k\|^2 \\
&= A_k g^*(\vv)  +  \frac{1}{2}\|\vv - \vy_0\|^2 - \frac{1+ \gamma A_k}{2}\|\vv - \vy_k\|^2,   %
\end{align*}
as claimed.

Bounding $\phi_k(\vx_k)$ can be done using the same sequence of arguments and is thus omitted. 
\end{proof}

\begin{lemma}\label{lem:psi-phi-lower-general}
In Algorithm \ref{alg:pda},  we have: $\forall (\vu,\vv)\in 
{\gX \times \gY}$
and $k\ge 1,$ 
\begin{eqnarray}
\psi_k(\vy_k) &\ge&  \psi_{k-1}(\vy_{k-1})  +  \frac{1+\gamma A_{k-1}}{2}\|\vy_k - \vy_{k-1}\|^2  + a_k \big(g^*(\vy_k) + \langle - \mB \vx_k, \vy_k - \vv\rangle\big) \nonumber  \\
&& + a_k\langle \mB(\vx_k - \vx_{k-1}), \vy_k - \vv\rangle - a_{k-1}   \langle \mB(\vx_{k-1} - \vx_{k-2}), \vy_{k-1}  - \vv \rangle  \nonumber   \\
&&-a_{k-1}\langle \mB(\vx_{k-1} - \vx_{k-2}), \vy_k - \vy_{k-1}\rangle, \notag\\
\phi_k(\vx_k) &\ge& \phi_{k-1}(\vx_{k-1}) + \frac{1+\sigma A_{k-1}}{2}\|\vx_k - \vx_{k-1}\|^2 + a_k\big(\langle \vx_k - \vu, \mB^T \vy_k\rangle + \ell(\vx_k)\big).  \notag%
\end{eqnarray}
\end{lemma}
\begin{proof}
By the definition of $\psi_k(\vy_k)$, the  fact that $\vy_{k-1}$ is optimal for $\psi_{k-1}$, and the $1+\gamma A_{k-1}$-strong convexity of $\psi_{k-1}(\vy_{k-1})$, we have the following:  $\forall k \ge 1$,
\begin{eqnarray}
\psi_k(\vy_k) &=& \psi_{k-1}(\vy_k)  + a_k( \langle -\mB \bar{\vx}_{k-1}, \vy_k - \vv\rangle + g^*(\vy_k) ) \nonumber  \\
&\ge& \psi_{k-1}(\vy_{k-1}) + \frac{1+\gamma A_{k-1}}{2}\|\vy_k - \vy_{k-1}\|^2 + a_k( \langle - \mB \bar{\vx}_{k-1}, \vy_k - \vv\rangle + g^*(\vy_k) ).  \label{eq:psi-1}
\end{eqnarray}
On the other hand, by the definition of $\bar{\vx}_{k-1}$, we also have 
\begin{align}
\langle - \mB \bar{\vx}_{k-1}, \vy_k - \vv\rangle 
=\;& \langle - \mB (\bar{\vx}_{k-1} - \vx_k), \vy_k - \vv\rangle + \langle - \mB \vx_k, \vy_k - \vv\rangle  \nonumber\\
=\;& \langle \mB(\vx_k - \vx_{k-1}), \vy_k - \vv \rangle -  \frac{a_{k-1}}{a_k} \langle \mB(\vx_{k-1} - \vx_{k-2}), \vy_k - \vv \rangle +  \langle - \mB \vx_k, \vy_k - \vv\rangle \nonumber\\
=\;& \langle \mB(\vx_k - \vx_{k-1}), \vy_k - \vv\rangle -  \frac{a_{k-1}}{a_k}   \langle \mB(\vx_{k-1} - \vx_{k-2}), \vy_{k-1}  - \vv \rangle \nonumber \\
&-\frac{a_{k-1}}{a_k}  \langle \mB(\vx_{k-1} - \vx_{k-2}), \vy_k - \vy_{k-1}\rangle +  \langle - \mB \vx_k, \vy_k - \vv\rangle. \label{eq:psi-phi}
\end{align} 
Hence, combining Eqs.~\eqref{eq:psi-1} and \eqref{eq:psi-phi}, we have 
\begin{eqnarray}
\psi_k(\vy_k) &\ge&  \psi_{k-1}(\vy_{k-1})  +  \frac{1+\gamma A_{k-1}}{2}\|\vy_k - \vy_{k-1}\|^2  + a_k g^*(\vy_k)   \nonumber\\
&& + a_k\langle \mB(\vx_k - \vx_{k-1}), \vy_k - \vv\rangle - a_{k-1}   \langle \mB(\vx_{k-1} - \vx_{k-2}), \vy_{k-1}  - \vv \rangle \nonumber \\
&&-a_{k-1}\langle \mB(\vx_{k-1} - \vx_{k-2}), \vy_k - \vy_{k-1}\rangle + a_k \langle - \mB \vx_k, \vy_k - \vv\rangle. \notag
\end{eqnarray}

Similarly, by the definition of $\phi_k,$ the $(1+\sigma A_{k-1})$-strong convexity of $\phi_{k-1}$, and $\vx_{k-1} = \argmin_{\vx \in \sR^d}\phi_{k-1}(\vx)$,  we have $\forall k \ge 1$ that
\begin{align*}
\phi_k(\vx_k) & = \phi_{k-1}(\vx_k) + a_k (\langle \vx_k - \vu, \mB^T \vy_k\rangle + \ell(\vx_k)) \\
& \ge \phi_{k-1}(\vx_{k-1}) + \frac{1+\sigma A_{k-1}}{2}\|\vx_k - \vx_{k-1}\|^2 + a_k(\langle \vx_k - \vu, \mB^T \vy_k\rangle + \ell(\vx_k)), \notag%
\end{align*}
completing the proof. 
\end{proof}

We are now ready to prove Theorem~\ref{thm:general}. 
\thmgeneral*
\begin{proof}
Applying Lemma~\ref{lem:psi-phi-lower-general}, we have 
\begin{equation} \label{eq:phi-psi-2}
\begin{aligned}
\psi_k(\vy_k) + \phi_k(\vx_k) 
 \ge\; &  \psi_{k-1}(\vy_{k-1}) +  \phi_{k-1}(\vx_{k-1})\\
 &+ a_k \langle \mB(\vx_k - \vx_{k-1}), \vy_k - \vv \rangle - a_{k-1}\langle \mB(\vx_{k-1} - \vx_{k-2}), \vy_{k-1}  - \vv \rangle     \\ 
&+  \bigg[\frac{1+\gamma A_{k-1}}{2}\|\vy_k - \vy_{k-1}\|^2 + \frac{1+\sigma A_{k-1}}{2}\|\vx_k - \vx_{k-1}\|^2 \\
&- a_{k-1}\langle \mB(\vx_{k-1} - \vx_{k-2}), \vy_k - \vy_{k-1}\rangle\bigg]_1 \\
& + \big[a_k (g^*(\vy_k) + \ell(\vx_k)  - \langle \mB \vu, \vy_k\rangle + \langle \vx_k, \mB^T \vv\rangle)\big]_2. 
\end{aligned} 
\end{equation}
The terms from the first line in this inequality give a recursive relationship for the sum of estimate sequences $\psi_k + \phi_k$. The terms from the second line telescope. Thus, we only need to focus on bounding the terms inside $[\cdot]_1$ and $[\cdot]_2$. 

Observe that, by the definition \eqref{eq:Guv} of $\Gap^{\vu, \vv}(\vx, \vy)$, we have   
\begin{equation}\label{eq:[.]_2}
    [\cdot]_2 = a_k \big(\Gap^{\vu, \vv}(\vx_k, \vy_k) + g^*(\vv) + \ell(\vu)\big).
\end{equation}

To bound $[\cdot]_1,$ observe first that
\begin{align*}
a_{k-1}\langle \mB(\vx_{k-1} - \vx_{k-2}), \vy_k - \vy_{k-1}\rangle %
&\le  a_{k-1}\|\mB\|\|\vx_{k-1} - \vx_{k-2}\| \| \vy_k - \vy_{k-1} \|  \\ 
 &\leq R a_{k-1}\|\vx_{k-1} - \vx_{k-2}\| \| \vy_k - \vy_{k-1} \|,
 \end{align*}
 where we have used Cauchy-Schwarz inequality, definition of the operator norm, and $\|\mB\| \leq R$ (which holds by Assumption~\ref{ass:general}). Applying Young's inequality, we have for all $k \ge 2$ that
 \begin{align}
\notag
a_{k-1}\langle \mB(\vx_{k-1} - \vx_{k-2}), \vy_k - \vy_{k-1}\rangle &\le \frac{R^2 {a_{k-1}}^2}{2(1+\gamma A_{k-1})}\|\vx_{k-1} - \vx_{k-2}\|^2 + \frac{1+\gamma A_{k-1}}{2}\|\vy_k - \vy_{k-1}\|^2   \\ 
\notag
&\le  \frac{R^2 {a_{k-1}}^2}{2(1+\gamma A_{k-2})}\|\vx_{k-1} - \vx_{k-2}\|^2 + \frac{1+\gamma A_{k-1}}{2}\|\vy_k - \vy_{k-1}\|^2\nonumber   \\ 
\notag 
&= \frac{1+\sigma A_{k-2}}{4}\|\vx_{k-1} - \vx_{k-2}\|^2 + \frac{1+\gamma A_{k-1}}{2}\|\vy_k - \vy_{k-1}\|^2, \\
\label{eq:K-1}
&\leq \frac{1+\sigma A_{k-2}}{2}\|\vx_{k-1} - \vx_{k-2}\|^2 + \frac{1+\gamma A_{k-1}}{2}\|\vy_k - \vy_{k-1}\|^2, 
\end{align}
where the second line is by $A_{k-1}\geq A_{k-2},$ $\forall k,$ and the third line is by the definition of $a_{k-1}$ in Algorithm~\ref{alg:pda}. Hence, we have for $k \ge 2$ that
\begin{equation}\label{eq:[.]_1}
    [\cdot]_1 \geq \frac{1+\sigma A_{k-1}}{2}\|\vx_k - \vx_{k-1}\|^2  - \frac{1+\sigma A_{k-2}}{2}\|\vx_{k-1} - \vx_{k-2}\|^2.
\end{equation}
For $k=1$, we have
\begin{align}
 [\cdot]_1 & = \frac{1 + \gamma A_0}{2} \|\vy_1 - \vy_0 \|^2 + \frac{1+\sigma A_0}{2} \|\vx_1 - \vx_0 \|^2  - a_0 \langle \mB (\vx_0 - \vx_{-1}), \vy_1-\vy_0 \rangle \notag \\
 & = \frac12 \|\vy_1 - \vy_0 \|^2 + \frac12 \| \vx_1 - \vx_0 \|^2. \label{eq:si9}
\end{align}
When we sum $[\cdot]_1$ over $1,2,\dotsc,k$, we obtain from \eqref{eq:[.]_1} and \eqref{eq:si9}  that the telescoped sum is bounded below by
\begin{align}
\notag
   & \frac{1+\sigma A_{k-1}}{2} \| \vx_k - \vx_{k-1}\|^2 - \frac{1+\sigma A_0}{2} \| \vx_1 - \vx_0 \|^2 + \frac12 \|\vy_1 - \vy_0 \|^2 + \frac12 \| \vx_1 - \vx_0 \|^2  \\
    \label{eq:si8}
   & = \frac{1+\sigma A_{k-1}}{2} \| \vx_k - \vx_{k-1}\|^2 + \frac12 \| \vy_1 - \vy_0 \|^2.
\end{align}

Thus, by combining Eqs.~\eqref{eq:phi-psi-2}--\eqref{eq:[.]_2}, telescoping from $1$ to $k$, and using \eqref{eq:si8}, we have
\begin{align}
\psi_k(\vy_k) + \phi_k(\vx_k) &\ge \psi_{0}(\vy_{0}) +  \phi_{0}(\vx_{0})  + a_k \langle \mB(\vx_k - \vx_{k-1}), \vy_k - \vv \rangle - a_{0}\langle \mB(\vx_{0} - \vx_{-1}), \vy_{0}  - \vv \rangle  \nonumber   \\ 
& \quad + \frac{1+\sigma A_{k-1}}{2}\|\vx_k - \vx_{k-1}\|^2 + \frac12 \| \vy_1 - \vy_0 \|^2 \nonumber  \\
& \quad + \sum_{i=1}^k a_i\big(\Gap^{\vu, \vv}(\vx_i, \vy_i) + g^*(\vv) + \ell(\vu)\big) \notag\\
& \ge a_k \langle \mB(\vx_k - \vx_{k-1}), \vy_k - \vv \rangle + \frac{1+\sigma A_{k-1}}{2}\|\vx_k - \vx_{k-1}\|^2  \nonumber \\
& \quad + \sum_{i=1}^k a_i \Gap^{\vu, \vv}(\vx_i, \vy_i) + A_k(g^*(\vv) + \ell(\vu)), \label{eq:recur} %
\end{align}
where we have used $\phi_0(\vx_0) = \psi_0(\vy_0) = 0$ and $\vx_0 = \vx_{-1},$ which holds by assumption.
By rearranging Eq.~\eqref{eq:recur} and using the bounds on $\phi_k (\vx_k)$ and  $\psi_k(\vy_k)$ from Lemma~\ref{lem:psi-phi-upper-general}, we have
\begin{equation}\label{eq:gap-almost-there}
    \begin{aligned}
        \sum_{i=1}^k a_i \Gap^{\vu, \vv}(\vx_i, \vy_i) \leq \; & -a_k \langle \mB(\vx_k - \vx_{k-1}), \vy_k - \vv \rangle - \frac{1+\sigma A_{k-1}}{2}\|\vx_k - \vx_{k-1}\|^2 \\
        &+\frac{1}{2}\|\vv - \vy_0\|^2 - \frac{1+ \gamma A_k}{2}\|\vv - \vy_k\|^2  \\
&  + \frac{1}{2}\|\vu - \vx_0\|^2 - \frac{1+ \sigma A_k}{2}\|\vu - \vx_k\|^2.
    \end{aligned}
\end{equation}
By using the same sequence of arguments leading to Eq.~\eqref{eq:K-1}, but with $\vv$ replacing $\vy_{k-1}$ and $k+1$ replacing $k$, we have that the following bound holds for all $k \ge 1$:
\begin{align*}
    -a_k \langle \mB(\vx_k - \vx_{k-1}), \vy_k - \vv \rangle \leq \;& \frac{1+\sigma A_{k-1}}{2}\|\vx_{k} - \vx_{k-1}\|^2 + \frac{1+\gamma A_{k}}{4}\|\vy_k - \vv\|^2.
\end{align*}
By combining  with Eq.~\eqref{eq:gap-almost-there}, we obtain
\begin{equation}\label{eq:pda2-final-bnd}
\begin{aligned}
    \sum_{i=1}^k a_i \Gap^{\vu, \vv}(\vx_i, \vy_i) \leq \;& \frac12 \|\vu - \vx_0\|^2 + \frac12 \|\vv - \vy_0\|^2\\
    &- \frac{1 + \sigma A_k}{2}\|\vu - \vx_k\|^2 - \frac{1 + \gamma A_k}{4}\| \vv-\vy_k\|^2. 
\end{aligned}
\end{equation}

To complete bounding the primal-dual gap, it remains to observe that for any fixed $\vu, \vv,$ $\Gap^{\vu, \vv}(\vx, \vy)$ is separable and convex in $\vx, \vy.$ Thus, by the definition of $\vxt_k, \vyt_k$ and Jensen's inequality, we have that 
\[
\Gap^{\vu, \vv}(\vxt_k, \vyt_k) \leq \frac{1}{A_k}\sum_{i=1}^k a_i \Gap^{\vu, \vv}(\vx_i, \vy_i) \leq \frac{\|\vu - \vx_0\|^2 + \|\vv - \vy_0\|^2}{2A_k}.
\]

To bound $\|\vx_k - \vx^*\|^2 + \|\vy_k - \vy^*\|^2,$ note that for $(\vu, \vv) = (\vx^*, \vy^*)$ we must have $\Gap^{\vu, \vv}(\vx, \vy) \geq 0,$ for all $(\vx, \vy)$, since $(\vx^*, \vy^*)$ is a primal-dual solution. Thus, rearranging Eq.~\eqref{eq:pda2-final-bnd}, we arrive at the claimed bound
$$
    (1+\sigma A_k)\|\vx_k - \vx^*\|^2 + \frac{1+\gamma A_k}{2}\|\vy_k - \vy^*\|^2 \leq \|\vx_0 - \vx^*\|^2 + \|\vy_0 - \vy^*\|^2.
$$

To complete the proof, it remains to bound the growth of $A_k.$ We do this using the relationship $A_k - A_{k-1} = a_k = \frac{ \sqrt{(1+\sigma A_{k-1})(1+\gamma A_{k-1})}}{\sqrt{2}R}$. When either $\gamma = 0$ or $\sigma = 0,$ the bound follows by applying Lemma~\ref{lemma:es-seq-growth}. If $\sigma > 0$ and $\gamma > 0$, then we have $A_k - A_{k-1} = a_k \ge  \frac{\sqrt{\sigma\gamma}}{\sqrt{2}R}A_{k-1},$ which leads to 
$$A_k = A_1 \Big(1 +  \frac{\sqrt{\sigma\gamma}}{\sqrt{2}R}\Big)^{k-1} = \frac{1}{\sqrt{2}R}\Big(1 +  \frac{\sqrt{\sigma\gamma}}{\sqrt{2}R}\Big)^{k-1},$$
as claimed.
\end{proof}

\section{Omitted Proofs from Section~\ref{sec:vrpdada} }\label{appx:omitted-proofs-vr}

We start by showing the following identity for the initial estimate sequence $\psi_1(\vy_1) + \phi_1(\vx_1)$.
\begin{lemma} \label{lem:init}
For the initialization steps (Lines~2-7 of Algorithm \ref{alg:vrpda}), we have, for all $(\vu,\vv) \in \gX \times \gY$ that
\begin{align*}
    \psi_1(\vy_1) + {\phi}_1(\vx_1) =\;& \frac{n}{2}\|\vy_1 - \vy_0\|^2 + \frac{n}{2}\|\vx_1 - \vx_0\|^2\\
    &+ a_1 \big(\Gap^{\vu, \vv}(\vx_1, \vy_1) + g^*(\vv) + \ell(\vu) + \langle \vx_1 - \vx_0, \mB^T(\vy_1 - \vv)\rangle\big).
\end{align*}
\end{lemma}
\begin{proof}
By the definition of $\tilde{\psi}_1(\cdot)$ and $\tilde{\phi}_1(\cdot)$, we have
\begin{align*}
\tilde{\psi}_1(\vy_1) &=
\frac{1}{2}\|\vy_1 - \vy_0\|^2 +  \tilde{a}_1 (\langle - \mB \vx_0,\vy_1 - \vv\rangle +  g^*(\vy_1)), \\
\tilde{\phi}_1(\vx_1) &=
\frac{1}{2}\|\vx_1 - \vx_0\|^2 +  \tilde{a}_1 (\langle {\vx}_1 - \vu, \mB^T\vy_1\rangle+\ell(\vx_1)). 
\end{align*}
Thus, using the definition \eqref{eq:Guv} of $\Gap^{\vu,\vv}$, we have: 
\begin{align*}
\tilde{\psi}_1(\vy_1) + \tilde{\phi}_1(\vx_1)  
& =  \frac{1}{2}\|\vy_1 - \vy_0\|^2 + \frac{1}{2}\|\vx_1 - \vx_0\|^2 \\
& \quad + \tilde{a}_1( g^*(\vy_1)  + \ell(\vx_1)  - \langle\mB\vu, \vy_1\rangle + \langle \vx_1, \mB^T\vv\rangle   + \langle \vx_1 - \vx_0, \mB^T(\vy_1 - \vv)\rangle) \\
&= \frac{1}{2}\|\vy_1 - \vy_0\|^2 + \frac{1}{2}\|\vx_1 - \vx_0\|^2  \\
& \quad + \tilde{a}_1( \Gap^{\vu, \vv}(\vx_1, \vy_1) + g^*(\vv) + \ell(\vu)   + \langle \vx_1 - \vx_0, \mB^T(\vy_1 - \vv)\rangle),
\end{align*}
We have by definition that $\psi_1 = n \tilde{\psi}_1,$ $\phi_1 = n \tilde{\phi}_1,$ and $a_1 = n\tilde{a}_1$, so the result follows when we multiply both sides of this expression by $n$.
\end{proof}

The following two lemmas now bound the growth of $\psi_k(\vy_k) + \phi_k(\vx_k)$ below and above, and are the main technical lemmas used in proving Theorem~\ref{thm:vr}.

\begin{lemma}\label{lem:psi-phi-upper}
For all steps of Algorithm \ref{alg:vrpda} with $k\ge 2$, we have, $\forall (\vu, \vv)\in\gX\times\gY,$ 
\begin{eqnarray}
\psi_k(\vy_k) &\le&   \sum_{i=2}^k a_i g_{j_i}^*(v_{j_i})  + a_1 g^*(\vv)  + \frac{n}{2}\|\vv - \vy_0\|^2 - \frac{n}{2}\|\vv - \vy_k\|^2,     \notag\\
\phi_k(\vx_k) &\le& A_k \ell(\vu) + \frac{n}{2}\|\vu - \vx_0\|^2 - \frac{n + \sigma A_k}{2}\|\vu - \vx_k\|^2.\notag%
\end{eqnarray}
\end{lemma}
\begin{proof}
By the definition of $\psi_k(\vy)$ in Algorithm \ref{alg:vrpda}, it follows that, $\forall k\ge 2,$ 
\begin{equation}\label{eq:vrpda-psi_k-def}
\begin{aligned}
\psi_k(\vy) =\; &\sum_{i=2}^k a_i( -\vb_{j_i}^T \bar{\vx}_{i-1} (y_{j_i} - v_{j_i}) + g_{j_i}^*(y_{j_i}) ) \\
&+ \left( \frac{n}{2}\|\vy - \vy_0\|^2 + a_1 (\langle - \mB \vx_0,\vy - \vv\rangle +  g^*(\vy))\right)
\end{aligned}
\end{equation}
and
\begin{equation}\label{eq:vrpda-phi_k-def}
\begin{aligned}
\phi_k(\vx) =\;& \sum_{i=2}^k a_i ( \langle \vx - \vu, \vz_{i-1} + (y_{i, j_i} - y_{i-1, j_i})\vb_{j_i} \rangle + \ell(\vx))         \\
&+ \Big(\frac{n}{2}\|\vx - \vx_0\|^2 + a_1 (\langle {\vx} - \vu, \mB^T\vy_1\rangle + \ell(\vx))\Big).
\end{aligned}
\end{equation}

By the first-order optimality condition in the definition of $\vy_k$ (see Algorithm \ref{alg:vrpda}), it follows that there exists 
$\vg_{g^*}(\vy_k)\in \partial g^*(\vy_k)$ such that 
\begin{eqnarray}\notag
\sum_{i=2}^k a_i( -\vb_{j_i}^T \bar{\vx}_{i-1} +(g_{j_i}^*)'(y_{k, j_i}))\ve_{j_i} + n(\vy_k - \vy_0) + a_1(-\mB\vx_0  + \vg_{g^*}(\vy_k))= \vzero,
\end{eqnarray}
where $\ve_{j_i}$ denotes the ${j_i}^{\mathrm{th}}$ standard basis vector (i.e., a vector whose element $j_i$ equals one, while all the remaining elements are zero) and $(g_{j_i}^*)' \in\partial g_{j_i}^*(y_{k, j_i})$ denotes the ${j_i}^{\mathrm{th}}$ element of $\vg_{g^*}(\vy_k))$. 
By rearranging this expression, we obtain
\begin{equation}\label{eq:vrpda-1st-order-opt}
    - a_1 \mB\vx_0  - \sum_{i=2}^k a_i \vb_{j_i}^T \bar{\vx}_{i-1}\ve_{j_i}     = - \sum_{i=2}^k a_i(g_{j_i}^*)'(y_{k, j_i})\ve_{j_i} -  n(\vy_k - \vy_0) - a_1\vg_{g^*}(\vy_k).  
\end{equation}

By setting $\vy=\vy_k$ in \eqref{eq:vrpda-psi_k-def}, then substituting from \eqref{eq:vrpda-1st-order-opt}, we obtain
\begin{align*}
\psi_k(\vy_k) =\;& \sum_{i=2}^k a_i (\langle (g_{j_i}^*)'(y_{k, j_i}), v_{j_i} - y_{k, j_i}\rangle + g_{j_i}^*(y_{k, j_i})) +  a_1(\langle \vg_{g^*}(\vy_k), \vv - \vy_k\rangle + g^*(\vy_k))  \nonumber  \\ 
&+ \langle n(\vy_k - \vy_0), \vv-\vy_k\rangle + \frac{n}{2}\|\vy_k - \vy_0\|^2  \nonumber  \\
\le\; &   \sum_{i=2}^k a_i g_{j_i}^*(v_{j_i})  + a_1 g^*(\vv)  + \frac{n}{2}\|\vv - \vy_0\|^2 - \frac{n}{2}\|\vv - \vy_k\|^2,    %
\end{align*}
where we have used $\langle \vg_{g^*}(\vy_k), \vv - \vy_k\rangle + g^*(\vy_k) \leq g^*(\vv)$ (by convexity of $g^*$) and $\langle n(\vy_k - \vy_0), \vv-\vy_k\rangle + \frac{n}{2}\|\vy_k - \vy_0\|^2 = \frac{n}{2}\|\vv - \vy_0\|^2 - \frac{n}{2}\|\vv - \vy_k\|^2.$

The bound on $\phi_k(\vx_k)$ follows from a similar sequence of arguments. 
By the first-order optimality in the definition of $\vx_k$, we have that there exists $\vg_{\ell}(\vx_k)\in\partial \ell(\vx_k)$ such that
\[
\sum_{i=2}^k a_i(\vz_{i-1} + (y_{i, j_i} - y_{i-1, j_i})\vb_{j_i}  + \vg_{\ell}(\vx_k))  + n(\vx_k - \vx_0) + a_1(\mB^T\vy_1 + \vg_{\ell}(\vx_k)) = \vzero,
\]
which rearranges to
\[
\sum_{i=2}^k a_i(\vz_{i-1} + (y_{i, j_i} - y_{i-1, j_i})\vb_{j_i}) +   a_1\mB^T\vy_1 = -\sum_{i=2}^k a_i \vg_{\ell}(\vx_k) - n(\vx_k-\vx_0) -a_1 \vg_{\ell}(\vx_k).
\]
By using this expression in Eq.~\eqref{eq:vrpda-phi_k-def} with $\vx=\vx_k$, we obtain
\begin{align*}
 \phi_k(\vx_k)  
=\; & \sum_{i=2}^k a_i (\langle \vg_{\ell}(\vx_k), \vu - \vx_k\rangle + \ell(\vx_k)) + a_1(\langle \vg_{\ell}(\vx_k), \vu - \vx_k\rangle + \ell(\vx_k)) \\ 
 & + \langle n(\vx_k - \vx_0), \vu - \vx_k\rangle + \frac{n}{2}\|\vx_k - \vx_0\|^2 \\ 
 \le\; & \sum_{i=2}^k a_i ( \ell(\vu) - \frac{\sigma}{2}\|\vu - \vx_k\|^2) + a_1( \ell(\vu) - \frac{\sigma}{2}\|\vu - \vx_k\|^2) + \frac{n}{2}\|\vu - \vx_0\|^2 - \frac{n}{2}\|\vu - \vx_k\|^2\\
 =\; & A_k \ell(\vu) + \frac{n}{2}\|\vu - \vx_0\|^2 - \frac{n + \sigma A_k}{2}\|\vu - \vx_k\|^2,
\end{align*}
where we have used $\langle \vg_{\ell}(\vx_k), \vu - \vx_k\rangle + \ell(\vx_k) \leq \ell(\vu) - \frac{\sigma}{2}\|\vu - \vx_k\|^2$ (by $\sigma$-strong convexity of $\ell$) and $\langle n(\vx_k - \vx_0), \vu - \vx_k\rangle + \frac{n}{2}\|\vx_k - \vx_0\|^2 = \frac{n}{2}\|\vu - \vx_0\|^2 - \frac{n}{2}\|\vu - \vx_k\|^2.$
The last line follows from the definition of $A_k$.
\end{proof}

\begin{lemma}\label{lem:psi-phi-lower}
For all steps of Algorithm \ref{alg:vrpda} with $k\ge 2$, taking expectation on all the randomness in the algorithm, %
we have for all $(\vu, \vv) \in \gX \times \gY$ that
\begin{eqnarray}
\E[\psi_k(\vy_k)] &\ge& \E\Big[\psi_{k-1}(\vy_{k-1}) + \frac{ n }{2}\|\vy_k - \vy_{k-1}\|^2  + a_k g^*_{j_k}(y_{k,j_k}) \nonumber \\ 
&&\quad\quad + a_k \langle \mB(\vx_k - \vx_{k-1}), \vy_k - \vv\rangle - a_{k-1}  \langle \mB(\vx_{k-1} - \vx_{k-2}), \vy_{k-1}  - \vv \rangle \nonumber \\
&& \quad\quad - n a_{k-1} \langle \mB(\vx_{k-1} - \vx_{k-2}), \vy_k - \vy_{k-1}\rangle + a_k  \langle - \mB \vx_k, \vy_k - \vv\rangle \nonumber  \\ 
&&\quad\quad  - (n-1)a_k\Big( \langle \mB ({\vx}_{k-1} - \vu), \vy_k - \vy_{k-1}\rangle  +
\langle \mB  \vu, \vy_k - \vy_{k-1}\rangle\Big)\Big],\notag\\ %
\E[\phi_k(\vx_k)] &\ge&\E\Big[\phi_{k-1}(\vx_{k-1}) + \frac{n+\sigma A_{k-1}}{2}\|\vx_k - \vx_{k-1}\|^2 \nonumber \\ 
&&   \quad\quad+ a_k(\langle \vx_k - \vu, \mB^T {\vy}_k \rangle + (n-1) \langle \vx_k  - \vu,  \mB^T(\vy_{k} -\vy_{k-1})\rangle+ \ell(\vx_k))\Big]. \notag%
\end{eqnarray}
\end{lemma}
\begin{proof}
By the definition of $\psi_k$, we have
\begin{align}
\nonumber
\psi_k(\vy_k) =\; & \psi_{k-1}(\vy_k)  + a_k( - \vb_{j_k}^T \bar{\vx}_{k-1}(y_{k, j_k} - v_{j_k}) + g^*_{j_k}(y_{k,j_k})) \\
\nonumber
= \; & \psi_{k-1}(\vy_{k-1}) + \big[\psi_{k-1}(\vy_k) - \psi_{k-1}(\vy_{k-1})\big]_1 \\
&+ \big[a_k( - \vb_{j_k}^T \bar{\vx}_{k-1}(y_{k, j_k} - v_{j_k}) + g^*_{j_k}(y_{k,j_k}) )\big]_2.  \label{eq:vr-psi-0}
\end{align}

To bound $\psi_k(\vy_k)$ in expectation and obtain the claimed bound, we need to bound the terms in $[\cdot]_1$ and $[\cdot]_2$. To do so, 
let $\gF_k$ be the natural filtration that contains all the randomness up to and including iteration $k.$ In what follows, we will use the tower property of conditional expectation, which guarantees $\E[\cdot] = \E[\E[\cdot|\gF_{k-1}]].$ %

To bound $[\cdot]_1,$ we use the definition of $\psi_{k-1}$ from Algorithm~\ref{alg:vrpda}, and the facts that $\vy_{k-1}$ is optimal for $\psi_{k-1}$ and that  $\psi_{k-1}$ is the sum of the $n$-strongly convex function $\psi_1$ with $k-2$ additional convex terms. We thus obtain
\begin{equation}\label{eq:vrpda-[]_1-bnd}
    [\cdot]_1 = \psi_{k-1}(\vy_k) - \psi_{k-1}(\vy_{k-1}) \geq \frac{n}{2}\|\vy_k - \vy_{k-1}\|^2.
\end{equation}

To bound $[\cdot]_2,$ observe that $\vy_k$ and $\vy_{k-1}$ only differ over the coordinate $j_k,$ which is chosen uniformly at random, independent of the history. Further, recall that $\mB = \frac{1}{n}[\vb_1, \vb_2, \ldots, \vb_n]^T$, and let $\mB_{-j_k}$ denote the matrix $\mB$ with its ${j_k}^{\mathrm{th}}$ row replaced by a zero vector. Then, we have  
\begin{align*}
\E\big[ - \vb_{j_k}^T \bar{\vx}_{k-1} (y_{k, j_k} - v_{j_k})| \gF_{k-1} \big]
&=\E\big[ \langle - n \mB \bar{\vx}_{k-1}, \vy_k - \vv\rangle + \langle n \mB_{-j_k} \bar{\vx}_{k-1}, \vy_k - \vv\rangle | \gF_{k-1}\big]   \\
&= \E\big[\langle - n \mB \bar{\vx}_{k-1}, \vy_k - \vv\rangle | \gF_{k-1} \big] +   \langle (n-1) \mB \bar{\vx}_{k-1}, \vy_{k-1} - \vv\rangle,%
\end{align*}
where the second equality follows from $\vy_k$ being equal to $\vy_{k-1}$ over all the coordinates apart from $j_k,$ and from $j_k$ being chosen uniformly at random. Taking expectations on both sides of the last equality and using $\langle\mB \bar{\vx}_{k-1}, \vy_{k} - \vv\rangle = \langle\mB \bar{\vx}_{k-1},  \vy_k - \vy_{k-1}\rangle + \langle\mB \bar{\vx}_{k-1}, \vy_{k-1} - \vv\rangle$ and the tower property of conditional expectation, we have
\begin{equation}\label{eq:vrpda-psi-bnd-err-1}
    \E\big[ - \vb_{j_k}^T \bar{\vx}_{k-1} (y_{k, j_k} - v_{j_k})\big] = \E\big[ \langle -  \mB \bar{\vx}_{k-1}, \vy_k - \vv\rangle\big]  - (n-1) \E\big[\langle \mB \bar{\vx}_{k-1}, \vy_k - \vy_{k-1}\rangle\big]. 
\end{equation}

To finish bounding $\E\big[ - \vb_{j_k}^T \bar{\vx}_{k-1} (y_{k, j_k} - v_{j_k})\big]$ (and, consequently, $\E[[\cdot]_2]$), we now proceed to bound the terms inside the expectations in Eq.~\eqref{eq:vrpda-psi-bnd-err-1}. 
First, adding an subtracting $\vx_k$ in the first inner product term and using the definition of $\vxb_{k-1},$ we have
\begin{align}
\langle - \mB \bar{\vx}_{k-1}, \vy_k - \vv\rangle  
=\; & \langle - \mB (\bar{\vx}_{k-1} - \vx_k), \vy_k - \vv\rangle - \langle \mB \vx_k, \vy_k - \vv\rangle\nonumber\\
=\; & \langle \mB(\vx_k - \vx_{k-1}), \vy_k - \vv \rangle -  \frac{a_{k-1}}{a_k}\langle \mB(\vx_{k-1} - \vx_{k-2}), \vy_k - \vv \rangle -  \langle  \mB \vx_k, \vy_k - \vv\rangle\nonumber\\
=\;& \langle \mB(\vx_k - \vx_{k-1}), \vy_k - \vv\rangle -  \frac{a_{k-1}}{a_k}   \langle \mB(\vx_{k-1} - \vx_{k-2}), \vy_{k-1}  - \vv \rangle \nonumber \\
&-\frac{a_{k-1}}{a_k}  \langle \mB(\vx_{k-1} - \vx_{k-2}), \vy_k - \vy_{k-1}\rangle - \langle  \mB \vx_k, \vy_k - \vv\rangle.  \label{eq:vr-psi-2}
\end{align} 
On the other hand, using the definition of $\vxb_{k-1},$ we also  have
\begin{align}
\langle \mB \bar{\vx}_{k-1}, \vy_k - \vy_{k-1}\rangle  
=\; &\langle \mB {\vx}_{k-1}, \vy_k - \vy_{k-1}\rangle  + \frac{a_{k-1}}{a_k}\langle \mB ({\vx}_{k-1} - \vx_{k-2}), \vy_k - \vy_{k-1}\rangle \nonumber  \\
=\; & \langle \mB ({\vx}_{k-1} - \vu), \vy_k - \vy_{k-1}\rangle  +
\langle \mB  \vu, \vy_k - \vy_{k-1}\rangle \notag\\
&+ \frac{a_{k-1}}{a_k}\langle \mB ({\vx}_{k-1} - \vx_{k-2}), \vy_k - \vy_{k-1}\rangle.  \label{eq:vr-psi-3}
\end{align}
Thus, combining Eqs.~\eqref{eq:vrpda-psi-bnd-err-1}-\eqref{eq:vr-psi-3} with the definition of $[\cdot]_2,$ we have:
\begin{equation}\label{eq:vrpda-[]_2-bnd}
\begin{aligned}
\E\big[[\cdot]_2\big] =  \E\big[& a_k g^*_{j_k}(y_{k, j_k}) + {a_k}\langle \mB(\vx_k - \vx_{k-1}), \vy_k - \vv\rangle -  {a_{k-1}}  \langle \mB(\vx_{k-1} - \vx_{k-2}), \vy_{k-1}  - \vv \rangle\\
&-  n {a_{k-1}}  \langle \mB(\vx_{k-1} - \vx_{k-2}), \vy_k - \vy_{k-1}\rangle -  {a_k}\langle \mB \vx_k, \vy_k - \vv\rangle\\
&- (n-1) a_k \big( \langle \mB ({\vx}_{k-1} - \vu), \vy_k - \vy_{k-1}\rangle  +
\langle \mB \vu, \vy_k - \vy_{k-1}\rangle\big)\big].
\end{aligned}
\end{equation}
The bound on $\psi_k(\vy_k)$ from the statement of the lemma now follows by combining Eq.~\eqref{eq:vr-psi-0} with Eqs.~\eqref{eq:vrpda-[]_1-bnd} and \eqref{eq:vrpda-[]_2-bnd}.

To bound $\phi_k(\vx_k),$ we use similar arguments to those we used above for bounding $\psi_k(\vy_k)$. In particular, from the definition of $\phi_k(\vx_k),$ and using that $\phi_{k-1}(\vx_{k-1})$ is $(n + \sigma A_{k-1})$-strongly convex and minimized at $\vx_{k-1},$ we have
\begin{equation}\label{eq:vr-phi-0}%
\begin{aligned}
\phi_k(\vx_k) \ge\; & \phi_{k-1}(\vx_{k-1}) + \frac{n+\sigma A_{k-1}}{2}\|\vx_k - \vx_{k-1}\|^2  \\
&+ a_k\big(\langle \vx_k - \vu, \vz_{k-1} + (y_{k, j_k} - y_{k-1, j_k})\vb_{j_k} \rangle + \ell(\vx_k)\big). 
\end{aligned}
\end{equation}
Since $\vy_{k}$ and $\vy_{k-1}$ only differ on their ${j_k}$ element, by the definition of $\vz_k$, we have $\vz_k = \vz_{k-1} + \mB^T(\vy_k-\vy_{k-1})$, so by a recursive argument it follows that  $\vz_i = \mB^T\vy_i$ for all $i=1,2,\dotsc,k$. Thus, we have
\begin{align*}
\vz_{k-1} + (y_{k, j_k} - y_{k-1, j_k})\vb_{j_k}  =& \mB^T \vy_{k-1} + n\mB^T(\vy_k - \vy_{k-1})  \\
=& \mB^T\vy_k + (n-1)\mB^T(\vy_k - \vy_{k-1}),  
\end{align*}
and consequently
\begin{align}
\langle \vx_k - \vu, \vz_{k-1} + (y_{k, j_k} - y_{k-1, j_k})\vb_{j_k}\rangle
=\; & \langle \vx_k - \vu, \mB^T {\vy}_k \rangle + (n-1) \langle \vx_k  - \vu,  \mB^T(\vy_{k} -\vy_{k-1})\rangle.  \label{eq:vr-phi-1}
\end{align}
To complete the proof, it remains to combine Eqs.~\eqref{eq:vr-phi-0} and \eqref{eq:vr-phi-1}. 
\end{proof}

Using Lemmas~\ref{lem:init}--\ref{lem:psi-phi-lower}, we are now ready to prove our main result. 
\mainthmvr*
\begin{proof}
Fix any $(\vu, \vv) \in \gX \times \gY$. 
By combining the bounds on $\psi_k(\vy_k)$ and $\phi_k(\vx_k)$ from Lemma \ref{lem:psi-phi-lower}, we have $\forall k \ge 2$ that 
\begin{equation}\label{eq:vr-phi-psi-2}
\begin{aligned}
\E[\psi_k(\vy_k) + \phi_k(\vx_k)] 
 \ge \E\Big[&  \psi_{k-1}(\vy_{k-1}) +  \phi_{k-1}(\vx_{k-1})   \\
& + a_k \langle \mB(\vx_k - \vx_{k-1}), \vy_k - \vv \rangle -  a_{k-1}\langle \mB(\vx_{k-1} - \vx_{k-2}), \vy_{k-1}  - \vv \rangle \\
&+ P_k + Q_k,
\Big]. 
\end{aligned} 
\end{equation}
where
\begin{equation}\label{eq:vr-P_k-1}
    \begin{aligned}
        P_k =\; & \frac{n}{2}\|\vy_k - \vy_{k-1}\|^2 + \frac{n+\sigma A_{k-1}}{2}\|\vx_k - \vx_{k-1}\|^2\\
        &- n a_{k-1}\langle \mB(\vx_{k-1} - \vx_{k-2}), \vy_k - \vy_{k-1}\rangle + (n-1) a_k \langle \vx_k - \vx_{k-1}, \mB^T(\vy_k - \vy_{k-1})\rangle, 
    \end{aligned}
\end{equation}
and
\begin{equation}\label{eq:vr-Q_k-1}
    Q_k = a_k (g^*_{j_k}(y_{k, j_k}) + \ell(\vx_k)  - \langle \mB \vu, \vy_k\rangle   +  \langle \vx_k, \mB^T \vv\rangle    - (n-1)\langle \mB  \vu, \vy_k - \vy_{k-1}\rangle).
\end{equation}
Observe that the first line Eq.~\eqref{eq:vr-phi-psi-2} gives the desired recursive relationship for the sum of estimate sequences, while the second line telescopes. Thus, we only need to focus on bounding $P_k$ and $Q_k$. 

To bound $P_k,$ we start by bounding the inner product terms that appear in it. Recall that $\vy_k$ and $\vy_{k-1}$ only differ on coordinate $j_k.$ We thus have
\begin{align*}
    \langle \mB(\vx_{k-1} - \vx_{k-2}), \vy_k - \vy_{k-1}\rangle
=\; & \frac{1}{n} \vb_{j_k}^T (\vx_{k-1} - \vx_{k-2})(y_{k, j_k} - y_{k-1, j_k})\\
\leq\; & \frac{1}{n} \Big(\frac{1}{2\alpha}\big(\vb_{j_k}^T (\vx_{k-1} - \vx_{k-2})\big)^2 + \frac{\alpha}{2}(y_{k, j_k} - y_{k-1, j_k})^2\Big),
\end{align*}
for any $\alpha > 0,$ by Young's inequality. Further, as, by Assumption~\ref{ass:vr}, $\max_{1\leq j\leq n}\|\vb_j\| \leq R'$, applying Cauchy-Schwarz inequality, we have that $\vb_{j_k}^T (\vx_{k-1} - \vx_{k-2}) \leq R' \|\vx_{k-1} - \vx_{k-2}\|,$ and, hence
\begin{equation}\label{eq:vr-P_k-2}
    \langle \mB(\vx_{k-1} - \vx_{k-2}), \vy_k - \vy_{k-1}\rangle \leq \frac{1}{n}\Big(\frac{R'^2}{2\alpha}  \|\vx_{k-1} - \vx_{k-2}\|^2 +  \frac{\alpha}{2} \|\vy_{k} - \vy_{k-1}\|^2\Big). 
\end{equation}
Similarly, $\forall \beta >0, $ 
\begin{equation}\label{eq:vr-P_k-3}
\langle \mB(\vx_{k} - \vx_{k-1}), \vy_k - \vy_{k-1}\rangle 
\ge -\frac{1}{n}\left(\frac{R'^2 }{2\beta}\|\vx_{k}  - \vx_{k-1}\|^2 + \frac{\beta}{2}  \|\vy_{k} - \vy_{k-1}\|^2\right).
\end{equation}
Thus, by combining Eqs.~\eqref{eq:vr-P_k-1}, \eqref{eq:vr-P_k-2}, and \eqref{eq:vr-P_k-3}, we have $\forall \alpha, \beta > 0$ that 
\begin{equation}\notag
    \begin{aligned}
        P_k \geq \; & \frac{n^2 - \alpha n a_{k-1} - \beta(n-1)a_k}{2n}\|\vy_k - \vy_{k-1}\|^2\\
        &+ \frac{n(n + \sigma A_{k-1})- (n-1)a_k R'^2/\beta }{2n}\|\vx_k - \vx_{k-1}\|^2 - \frac{a_{k-1}R'^2}{2\alpha}\|\vx_{k-1} - \vx_{k-2}\|^2.
    \end{aligned}
\end{equation}
Taking $\alpha = \frac{n}{2a_{k-1}}$, $\beta = \frac{n}{2a_k},$ and recalling that by our choice of step sizes in Algorithm~\ref{alg:vrpda}, $a_k \leq \frac{\sqrt{n(n+\sigma A_{k-1})}}{2R'},$ $\forall k \geq 2,$ we can further simplify the bound on $P_k$ to 
\begin{equation}\label{eq:vr-P_k-final}
    P_k \geq \frac{n + \sigma A_{k-1}}{4}\|\vx_k - \vx_{k-1}\|^2 - \frac{n + \sigma A_{k-2}}{4}\|\vx_{k-1} - \vx_{k-2}\|^2, 
\end{equation}
which telescopes. 
Combining the bound on $P_k$ from Eq.~\eqref{eq:vr-P_k-final} with the initial bound on $\psi_k(\vy_k) + \phi_k(\vx_k)$ from Eq.~\eqref{eq:vr-phi-psi-2} and telescoping, we have
\begin{equation}\label{eq:vr-psi-phi-4}
\begin{aligned}
\E[\psi_k(\vy_k) + \phi_k(\vy_k)]  
\ge\E\Big[ &\psi_{1}(\vx_{1}) +  \phi_{1}(\vy_{1})\\
&+ a_k \langle \mB(\vx_k - \vx_{k-1}), \vy_k - \vv \rangle - a_{1}\langle \mB(\vx_{1} - \vx_{0}), \vy_{1}  - \vv \rangle     \\
&  +  \frac{n+\sigma A_{k-1}}{4}\|\vx_k - \vx_{k-1}\|^2  -  \frac{n}{4}\|\vx_{1} - \vx_{0}\|^2 + \sum_{i=2}^k Q_i\Big].  
\end{aligned}
\end{equation}
We now proceed to bound $\E\big[\sum_{i=2}^k Q_i\big].$ 
Observe first that $g^*_{j_i}(y_{i, j_i}) = n(g^*(\vy_i) - \frac{1}{n}\sum_{j \neq j_i}g^*_j(y_{i-1, j})).$ Thus,
\begin{align*}
    \E\big[g^*_{j_i}(y_{i, j_i})\big] &= \E\Big[\E\Big[n\Big(g^*(\vy_i) - \frac{1}{n}\sum_{j \neq j_i}g^*_j(y_{i-1, j})\Big)\Big|\gF_{i-1}\Big]\Big]\\
    &= \E\big[n g^*(\vy_i) - (n-1)g^*(\vy_{i-1})\big],
\end{align*}
where $\gF_{i-1}$ is the natural filtration, containing all randomness up to and including iteration $i-1.$  Therefore, we can bound $\E\big[\sum_{i=2}^k Q_i\big]$ as follows: 
\begin{align}
\E\big[\sum_{i=2}^k Q_i\big] =\; &\E\left[ \sum_{i=2}^k a_i (n g^*(\vy_i)  - (n-1)g^*(\vy_{i-1}))\right] + \E\left[ \sum_{i=2}^k a_i \ell(\vx_i)\right]\nonumber \\
& + \E\left[ \sum_{i=2}^k a_i( - n \langle \mB \vu, \vy_i\rangle  + (n-1)\langle \mB \vu, \vy_{i-1}\rangle    +  \langle \vx_i, \mB^T \vv\rangle )    \right]\nonumber\\
=\; &\E\Big[ n a_k g^*(\vy_k)  +  \sum_{i=2}^{k-1} ( n a_{i}  - (n-1)a_{i+1})g^*(\vy_{i}) - (n-1) a_2 g^*(\vy_1)\Big] \nonumber   \\
&  + \E\Big[ -n a_k \langle\mB \vu, \vy_k\rangle +         \sum_{i=2}^{k-1} ( -n a_i + (n-1)a_{i+1})\langle \mB\vu, \vy_i\rangle   + (n-1)a_2\langle \mB \vu, \vy_1\rangle \Big]  \nonumber \\
&  + \E\left[ \sum_{i=2}^k a_i \ell(\vx_i) + \sum_{i=2}^{k}a_i \langle\vx_i, \mB^T \vv\rangle \right]. \label{eq:f-g} 
\end{align}
On the other hand, recall that, by Lemma~\ref{lem:init}, we have
\begin{align}
\psi_1(\vy_1) + {\phi}_1(\vx_1) 
=\; &  \frac{n}{2}\|\vy_1 - \vy_0\|^2 + \frac{n}{2}\|\vx_1 - \vx_0\|^2 \nonumber \\
& + a_1( g^*(\vy_1) + \ell(\vx_1) - \langle\mB \vu, \vy_1\rangle  + \langle \vx_1, \mB^T\vv\rangle   + \langle \vx_1 - \vx_0, \mB^T(\vy_1 - \vv)\rangle) \nonumber \\
=\;& \frac{n}{2}\|\vy_1 - \vy_0\|^2 + \frac{n}{2}\|\vx_1 - \vx_0\|^2 \nonumber \\
& + (n-1) a_2( g^*(\vy_1) - \langle\mB \vu, \vy_1\rangle ) \nonumber \\
& + a_1(\ell(\vx_1) + \langle \vx_1, \mB^T\vv\rangle   + \langle \vx_1 - \vx_0, \mB^T(\vy_1 - \vv)\rangle ),\label{eq:init-2}
\end{align}
where we have used the setting $a_2 = \frac{a_1}{n-1}$ of Algorithm~\ref{alg:vrpda}.
 
Thus, combining Eqs.~\eqref{eq:vr-psi-phi-4}--\eqref{eq:init-2}, we have 
\begin{align}
\E[\psi_k(\vy_k) + \phi_k(\vy_k)] 
\ge \;&\E\Bigg[ a_k \langle \mB(\vx_k - \vx_{k-1}), \vy_k - \vv \rangle +  \frac{n+\sigma A_{k-1}}{4}\|\vx_k - \vx_{k-1}\|^2  \nonumber \\
& \quad + \frac{n}{2}\|\vy_1 - \vy_0\|^2   + \frac{n}{4}\|\vx_{1} - \vx_{0}\|^2     \nonumber  \\   
&\quad+  n a_k g^*(\vy_k)  +  \sum_{i=2}^{k-1} ( n a_{i}  - (n-1)a_{i+1})g^*(\vy_{i})    \nonumber   \\
&  \quad  -n a_k \langle\mB \vu, \vy_k\rangle -         \sum_{i=2}^{k-1} (n a_i - (n-1)a_{i+1})\langle \mB\vu, \vy_i\rangle    \nonumber   \\
& \quad+  \sum_{i=1}^k a_i \ell(\vx_i) + \sum_{i=1}^{k}a_i \langle\vx_i, \mB^T \vv\rangle\Bigg]. \nonumber 
\end{align}
Using convexity of $g^*$ and $\ell$ (to apply Jensen's inequality) and the definitions of $\vxt_k, \vyt_k,$ it now follows that 
\begin{align}
\E[\psi_k(\vy_k) + \phi_k(\vy_k)] \ge\; & \E\Bigg[ a_k \langle \mB(\vx_k - \vx_{k-1}), \vy_k - \vv \rangle  +  \frac{n+\sigma A_{k-1}}{4}\|\vx_k - \vx_{k-1}\|^2  \nonumber \\
& \quad + \frac{n}{2}\|\vy_1 - \vy_0\|^2 + \frac{n}{4}\|\vx_{1} - \vx_{0}\|^2     \nonumber  \\   
&\quad+ A_k (g^*(\tilde{\vy}_k)  + \ell(\tilde{\vx}_k)  - \langle\mB \vu, \tilde{\vy}_k\rangle + \langle \tilde{\vx}_k, \mB^T\vv\rangle)\Bigg].\label{eq:vpp-2} 
\end{align}
Convexity can be used here because, due to our setting of $a_i$, we have $n a_{i} - (n-1)a_{i+1}\ge 0$, $na_k + \sum_{i=2}^{k-1}( n a_i - (n-1)a_{i+1}) = n a_2 + \sum_{i=3}^k a_i = A_k,$ as $n a_2 = a_1 + a_2$,  and   $\vxt_k, \vyt_k$ are chosen as
$$
\tilde{\vy}_k  = \frac{1}{A_k}\Big(na_k\vy_k+  \sum_{i=2}^{k-1} ( n a_i - (n-1)a_{i+1})\vy_i \Big),\; 
 \tilde{\vx}_k := \frac{1}{A_k}\sum_{i=1}^k a_k \vx_i.$$

Recalling the definition \eqref{eq:Guv} of $\Gap^{\vu, \vv}(\vx, \vv)$ and rearranging Eq.~\eqref{eq:vpp-2}, we have
\begin{align*}
    A_k \E[\Gap^{\vu, \vv}(\vxt_k, \vyt_k) ] \leq \; & \E\Big[\psi_k(\vy_k) + \phi_k(\vx_k) - A_k(g^*(\vv) + \ell(\vu))\\
    &- a_k \langle \mB(\vx_k - \vx_{k-1}), \vy_k - \vv \rangle  -  \frac{n+\sigma A_{k-1}}{4}\|\vx_k - \vx_{k-1}\|^2\\
    &-\frac{n}{2}\|\vy_1 - \vy_0\|^2 - \frac{n}{4}\|\vx_{1} - \vx_{0}\|^2\Big].
\end{align*}
To complete the proof, it remains to apply Lemma~\ref{lem:psi-phi-upper}, and simplify. In particular, as
\begin{eqnarray}
\E\Big[ \sum_{i=2}^k a_i g_{j_i}^*(v_{j_i})  + a_1 g^*(\vv)\Big] =  \sum_{i=2}^k a_i g^*(\vv) + a_1 g^*(\vv) = A_k g^*(\vv), 
\end{eqnarray}
we have 
\begin{align*}
\E[\psi_k(\vy_k) + \phi_k(\vy_k)]  %
\le \E\Bigg[ &A_k g^*(\vv)    + \frac{n}{2}\|\vv - \vy_0\|^2 - \frac{n}{2}\|\vv - \vy_k\|^2  \\
& + A_k \ell(\vu) + \frac{n}{2}\|\vu - \vx_0\|^2 - \frac{n + \sigma A_k}{2}\|\vu - \vx_k\|^2\Bigg], %
\end{align*}
which leads to
\begin{align*}
    A_k \E[\Gap^{\vu, \vv}(\vxt_k, \vyt_k) ] \leq  \E\Big[& \frac{n(\|\vu - \vx_0\|^2 + \|\vv - \vy_0\|^2)}{2} - \frac{n}{2}\|\vv - \vy_k\|^2 - \frac{n + \sigma A_k}{2}\|\vu - \vx_k\|^2 \\
    &- a_k \langle \mB(\vx_k - \vx_{k-1}), \vy_k - \vv \rangle  -  \frac{n+\sigma A_{k-1}}{4}\|\vx_k - \vx_{k-1}\|^2\\
    &-\frac{n}{2}\|\vy_1 - \vy_0\|^2 - \frac{n}{4}\|\vx_{1} - \vx_{0}\|^2\Big].
\end{align*}

Finally, we have from Young's inequality and the definition of $a_k$ that 
\begin{align*}
-a_k\langle \mB(\vx_k - \vx_{k-1}), \vy_k - \vv \rangle %
\le \; & a_k \|\mB\|\|\vx_k - \vx_{k-1}\| \|\vy_k - \vv\|   \\
\le\; & R' a_k \|\vx_k - \vx_{k-1}\| \|\vy_k - \vv\| \\
\le \; & \frac{R'^2 a_k^2}{n} \|\vx_k - \vx_{k-1}\|^2 + \frac{n}{4}\|\vv - \vy_k\|^2 \\ 
\le \;& \frac{ n+\sigma A_{k-1}}{4} \|\vx_k - \vx_{k-1}\|^2 + \frac{n}{4}\|\vv - \vy_k\|^2, %
\end{align*}
leading to
\begin{align*}
    A_k \E[\Gap^{\vu, \vv}(\vxt_k, \vyt_k) ] \leq  \E\Big[& \frac{n(\|\vu - \vx_0\|^2 + \|\vv - \vy_0\|^2)}{2} - \frac{n}{4}\|\vv - \vy_k\|^2 - \frac{n + \sigma A_k}{2}\|\vu - \vx_k\|^2 \Big].
\end{align*}
Similarly as in the proof of Theorem~\ref{thm:general}, it now follows that, $\forall (\vu, \vv) \in \gX \times \gY,$%
\begin{align*}
    \E[\Gap^{\vu, \vv}(\vxt_k, \vyt_k) ] \leq  \frac{n(\|\vu - \vx_0\|^2 + \|\vv - \vy_0\|^2)}{2A_k}
\end{align*}
and, as $\Gap^{\vx^*, \vy^*}(\vxt_k, \vyt_k) \geq 0$, we also have
\begin{align*}
    \E\Big[\frac{n}{4}\|\vy^* - \vy_k\|^2 + \frac{n + \sigma A_k}{2}\|\vx^* - \vx_k\|^2 \Big] \leq \frac{n(\|\vx^* - \vx_0\|^2 + \|\vy^* - \vy_0\|^2)}{2}.
\end{align*}
The bound on $A_k$ follows by applying Lemma~\ref{lemma:vr-A_k-growth} (Appendix~\ref{appx:growth-of-seqs}).
\end{proof}

\section{Growth of Sequences}\label{appx:growth-of-seqs}

We start first with a general lemma that is useful for bounding the convergence rate of both \pda~and \vrpda.

\begin{restatable}{lemma}{lemseqgrowth}\label{lemma:es-seq-growth}
Let $\{A_k\}_{k \geq 0}$ be a sequence of nonnegative real numbers such that $A_0 = 0$ and $A_k$ is defined recursively via $A_k = A_{k-1} +  \sqrt{c_1^2 + c_2A_{k-1}}$, where $c_1 > 0,$ and $c_2 \ge 0$.
Define $K_0 = \lceil\frac{c_2}{9c_1}\rceil.$ Then 
\begin{equation}\notag
A_k \geq
    \begin{cases}
        \frac{c_2}{9}\Big(k - K_0 + \max\Big\{{3\sqrt{\frac{c_1}{c_2}},\, 1\Big\}\Big)^2}, &\text{ if } c_2 > 0 \text{ and } k > K_0,\\
        c_1 k, &\text{ otherwise}.%
    \end{cases}
\end{equation}
\end{restatable}
\begin{proof}
As, by assumption, $c_2 \geq 0$, we have that $A_k \geq A_{k-1} + c_1,$ which, combined with $A_0 = 0,$ implies $A_k \ge c_1 k.$ 
Thus, we only need to focus on proving the stated bound in the case that $c_2 > 0$ and $k > K_0.$   

To prove the lemma, let us start by assuming that there exist $p>0, q>0, $ and $k_0\in\{1,2,\ldots\}$, such that 
\begin{equation}\label{eq:es-init}
\begin{aligned}
A_{k_0} &\ge p q^2, \text{ and } \\
A_{k-1} &\ge p(k-1-k_0 + q)^2, %
\end{aligned}
\end{equation}
for some $k \geq k_0 + 1$ (observe that the inequalities are consistent for $k-1 = k_0$). Then, by induction on $k,$ we can prove that $A_k \geq p(k-k_0 + q)^2$, for some specific $p, q, k_0.$ In particular:
\begin{eqnarray*}
A_k &\ge&  A_{k-1} +  \sqrt{c_2A_{k-1}} \\
&\ge&  p(k-1-k_0 + q)^2 + \sqrt{c_2 p(k-1-k_0 + q)^2} \\ 
&=& p(k-k_0 + q)^2 - 2p(k-k_0 + q) + p +\sqrt{c_2p}(k-1-k_0 +q) \\ 
&=& p(k-k_0 + q)^2 - \sqrt{p}(2\sqrt{p} - \sqrt{c_2}) (k-k_0) + p(1-2q) + \sqrt{c_2 p}(-1+q).  
\end{eqnarray*}

Let $c_2 = 9p$ and $q \ge 1$. Then, we have 
\begin{eqnarray}
A_k &\ge&  p(k-k_0 + q)^2 + p (k-k_0) + p(1-2q) + 3p(-1+q) \nonumber \\  
&\ge&p(k-k_0 + q)^2 + p(-1 + q) \nonumber  \\ 
&\ge& p(k-k_0 + q)^2,  \notag%
\end{eqnarray}
where we have used $k \geq k_0 + 1.$ 

It remains to show that we can choose $p, q, k_0$ that make the definition of $A_k$ from the statement of the lemma consistent with the assumption from Eq.~\eqref{eq:es-init}.  

If  $c_2\le  9 c_1,$ we have $A_1 = c_1 \ge \frac{c_2}{9} = p.$ Thus, to satisfy Eq.~\eqref{eq:es-init}, we can set $k_0 = K_0 = \lceil\frac{c_2}{9c_1}\rceil = 1$ and $q= 3\sqrt{\frac{c_1}{c_2}}.$ 
On the other hand, if $c_2 > 9 c_1,$ we have $A_1 = c_1 < \frac{c_2}{9} = p.$ Thus, by setting $k_0 = K_0 =  \lceil\frac{c_2}{9c_1}\rceil$ and $q=1 > 3\sqrt{\frac{c_1}{c_2}}$ and using that $A_k \geq c_1 k$ (argued at the beginning of the proof), we have $A_{k_0} \ge \frac{c_2}{9}=p. $ 
\end{proof}

We now examine the properties of the sequence $\{a_k\}_{k \geq 1}$ defined in Algorithm~\ref{alg:vrpda} %
and restated here:
\begin{equation}\label{eq:akdef}
a_1=\frac{n}{2R'}, \quad
a_2 = \frac{1}{n-1} a_1, \quad
a_k = \min \Big\{ \Big(1+\frac{1}{n-1} \Big) a_{k-1}, \frac{\sqrt{n(n+\sigma A_{k-1})}}{2R'} \Big\} \;\; \mbox{for $k \ge 3$.}
\end{equation}
We also examine growth properties of $\{A_k\}_{k\geq 1}$ defined by $A_k = \sum_{i=1}^k a_i$.

\begin{proposition}\label{prop:vr-burn-out}
Suppose that $n \geq 2$. Then there exists an index $k_0$ such that $A_k = \frac{n-1}{2R'}\Big(1 + \frac{1}{n-1}\Big)^k$ for all $k \leq k_0,$ where
\[
k_0 = \left\lceil \frac{\log B_{n,\sigma,R'}}{\log n - \log (n-1)} \right\rceil,
\]
and 
\[
B_{n,\sigma,R'} =  \frac{\sigma n (n-1)}{4 R'} +  \sqrt{\left(  \frac{\sigma n (n-1)}{4 R'}\right)^2 + n^2} \geq 
n\max\left\{1, \frac{\sigma (n-1)}{2 R'}\right\}.
\]
Further,
$$
    (n-1)\log B_{n,\sigma,R'} \leq k_0 \leq 1.1 (n-1) \log B_{n,\sigma,R'}+1,
$$
and we can conclude that the dependence of $k_0$ on $n$ is  $k_0 = \Omega(n \log n)$.
\end{proposition}
\begin{proof}
Cases $k = 1, 2$ can be verified by inspection, as $A_1 = a_1 = \frac{n}{2R'} = \frac{n-1}{2R'}\big(1 + \frac{1}{n-1}\big)$ and $A_2 = a_1 + a_2 = a_1 \big(1 + \frac{1}{n-1}\big) = \frac{n-1}{2R'}\big(1 + \frac{1}{n-1}\big)^2.$ Observe that also $a_2  = \frac{1}{2R'}\big(1 + \frac{1}{n-1}\big).$  
For $k > 2,$ as long as $a_{k-1} \big(1 + \frac{1}{n-1}\big) \leq \frac{\sqrt{n(n+\sigma A_{k-1})}}{2R'}$ for all successive iterates, %
we have that
\begin{align*}
    a_k = a_{k-1}\Big(1 + \frac{1}{n-1}\Big) = \frac{1}{2R'}\Big(1 + \frac{1}{n-1}\Big)^{k-1}. 
\end{align*}
As $A_k = \sum_{i=1}^k a_i,$ we have
\begin{align*}
    A_k &= \frac{n}{2R'} + \frac{1}{2R'}\sum_{i=2}^k \Big(1 + \frac{1}{n-1}\Big)^{i-1}\\
    &= \frac{n}{2R'} + \frac{1}{2R'}\Big(\frac{\big(1 + \frac{1}{n-1}\big)^{k} - 1}{1 + \frac{1}{n-1} - 1} - 1\Big)\\
    &= \frac{n-1}{2R'}\Big(1 + \frac{1}{n-1}\Big)^{k}.
\end{align*}
Now let $k_0$ be the first iteration for which 
\begin{equation} \label{eq:bu.2}
a_{k_0} \Big(1 + \frac{1}{n-1}\Big) > \frac{\sqrt{n(n+\sigma A_{k_0})}}{2R'}. 
\end{equation}
Since $k_0$ is the first such iteration, we also have (by the argument above) that $a_{k_0} = \frac{1}{2R'}\big(1 + \frac{1}{n-1}\big)^{k_0-1}$ and 
$A_{k_0} = \frac{n-1}{2R'}\Big(1 + \frac{1}{n-1}\Big)^{k_0}.$ 
Thus, by using these equalities and squaring both sides in \eqref{eq:bu.2}, we obtain
$$
    \frac{n^2}{(n-1)^2}\cdot\frac{1}{(2R')^2}\left(1 + \frac{1}{n-1}\right)^{2(k_0-1)} > \frac{n^2 + n\sigma \frac{n-1}{2R'}\left(1 + \frac{1}{n-1}\right)^{k_0}}{(2R')^2}.
$$
After simplifying the last expression, we have
\[
      \Big(1 + \frac{1}{n-1}\Big)^{2k_0} > n^2 \Big(1 + \frac{\sigma(n-1)}{2n R'}\Big(1 + \frac{1}{n-1}\Big)^{k_0}\Big),
\]
and we seek the smallest positive  integer $k_0$ that satisfies this property.
By introducing the notation $r = \left(1+\frac{1}{n-1}\right)^{k_0}$, we can write this condition as
\[
r^2 > n^2  \left( 1+ \frac{\sigma(n-1)}{2nR'} r \right)  \Leftrightarrow r^2 - \frac{\sigma n(n-1)}{2 R'} r - n^2 >0,
\]
from which we obtain (by solving the quadratic) that 
$r > B_{n,\sigma,R'}$, where 
$$
B_{n,\sigma,R'} = \frac{1}{2}\bigg(\frac{\sigma n(n-1)}{2 R'} + \sqrt{\Big(\frac{\sigma n(n-1)}{2 R'}\Big)^2 + 4 n^2}\bigg),
$$ 
which is identical to the definition of $B_{n,\sigma,R'}$ in the statement of the result.

Thus,  $k_0$ is the smallest integer such that
\[
\left(1+\frac{1}{n-1} \right)^{k_0} > B_{n,\sigma,R'},
\]
or, in other words, $k_0 = \lceil \kappa_0 \rceil$, where $\kappa_0$ satisfies
\[
\left(1+\frac{1}{n-1} \right)^{\kappa_0} = B_{n,\sigma,R'},
\]
which yields the main result when we take logs of both sides.

By using $\log (1+\delta) \in ((1/1.1) \delta, \delta)$ for $\delta \in(0, 0.21)$, we have
\[
\kappa_0 = \frac{\log B_{n,\sigma,R'}}{\log\left(1+\frac{1}{n-1}\right)} \in (1, 1.1) (n-1) \log B_{n,\sigma,R'}
\]
for $n \geq 2$. The final claim is immediate.
\end{proof}

\begin{proposition} \label{prop:switch}
For $k_0$ defined in Proposition~\ref{prop:vr-burn-out} and $n \geq 2$, we have for all $k \geq k_0$ that
\[
a_k \left( 1+\frac{1}{n-1} \right) > \frac{\sqrt{n(n+\sigma A_k)}}{2R'}.
\]
Thus, we have that 
\begin{equation}
    a_{k+1} = \frac{\sqrt{n(n+\sigma A_k)}}{2R'}, \quad \mbox{for all $k \geq k_0$.}
\end{equation}
\end{proposition}
\begin{proof}
Suppose for the purpose of contradiction that there is $k \geq k_0$ such that 
\begin{align} \label{eq:wr1}
a_k \left( 1+\frac{1}{n-1} \right) & > \frac{\sqrt{n(n+\sigma A_k)}}{2R'} \\
\label{eq:wr3}
a_{k+1} \left( 1+\frac{1}{n-1} \right) & \le  \frac{\sqrt{n(n+\sigma A_{k+1})}}{2R'}.
\end{align}
It follows from \eqref{eq:wr1} that
\begin{equation}
    \label{eq:wr2}
    a_{k+1} = \frac{\sqrt{n(n+\sigma A_k)}}{2R'}.
\end{equation}
By squaring both sides of  \eqref{eq:wr3} and using \eqref{eq:wr2} and  $A_{k+1} = A_k+a_k$, we have
\begin{align*}
    \left( 1+\frac{1}{n-1} \right)^2 {a_{k+1}}^2 & \le \frac{n(n+\sigma A_{k+1})}{(2R')^2 } \\
    \Leftrightarrow \;\; \left(\frac{n}{n-1} \right)^2 \frac{n(n+\sigma A_{k})}{(2R')^2} & \le 
    \frac{n(n+\sigma A_k + \sigma a_{k+1})}{(2R')^2} \\
    \Leftrightarrow \;\;  \left( \frac{n^2}{(n-1)^2} - 1 \right) \frac{n(n+\sigma A_{k})}{(2R')^2} & \le
    \frac{n \sigma a_{k+1}}{(2R')^2} \\
    \Leftrightarrow \;\;  %
    \frac{2n-1}{(n-1)^2}
    {a_{k+1}}^2 & \le \frac{n \sigma}{(2R')^2} a_{k+1} \\
    \Leftrightarrow \;\; a_{k+1} & \le  \frac{n(n-1)^2 \sigma}{(2R')^2(2n-1)},
\end{align*}
and it follows by taking logs of both sides of the last expression that 
\begin{equation}
    \label{eq:wr4}
    \log a_{k+1} \le \log \Big(\frac{1}{2R'}\Big) + \log \Big(\frac{n-1}{2n-1}\Big) + \log \Big(\frac{\sigma n(n-1)}{2R'}\Big).
\end{equation}
We now obtain a {\em lower} bound on $a_{k_0}$. Since, as noted in the proof of Proposition~\ref{prop:vr-burn-out}, we have $a_{k_0} = \frac{1}{2R'} \left( 1+ \frac{1}{n-1} \right)^{k_0-1}$, we have, using the definitions of $k_0$ and $B_{n,\sigma,R'}$ in the statement of Proposition~\ref{prop:vr-burn-out} that 
\begin{align*}
\log a_{k_0} & = \log \Big(\frac{1}{2R'}\Big) + (k_0-1) \log \Big(\frac{n}{n-1}\Big) \\
&= \log \Big(\frac{1}{2R'}\Big) - \log \Big(\frac{n}{n-1}\Big) + k_0 \log \Big(\frac{n}{n-1}\Big) \\
& \ge \log \Big(\frac{1}{2R'}\Big) - \log \Big(\frac{n}{n-1}\Big) + \log (B_{n,\sigma,R'}) \\
& > \log \Big(\frac{1}{2R'}\Big) - \log \Big(\frac{n}{n-1}\Big) + \log \Big(\frac{\sigma n(n-1)}{2R'}\Big) \\
\end{align*}
Now for $n \ge 2$, we have
\[
- \log \Big(\frac{n}{n-1}\Big) = \log \Big(\frac{n-1}{n}\Big) \ge \log \Big(\frac{n-1}{2n-1}\Big),
\]
so by substituting in the last expression above, we obtain
\begin{equation}
    \label{eq:wr5}
    \log a_{k_0} >\log \Big(\frac{1}{2R'}\Big) + \log \Big(\frac{n-1}{2n-1}\Big) + \log \Big(\frac{\sigma n(n-1)}{2R'}\Big).
\end{equation}
By comparing Eq.~\eqref{eq:wr4} and Eq.~\eqref{eq:wr5}, we see that $a_{k+1} < a_{k_0}$ which (since $\{a_i\}_{i \geq 1}$ is a monotonically increasing sequence) implies that $k+1<k_0$, which contradicts our choice of $k$. Thus, no such $k$ exists, and our proof is complete.
\end{proof}

Using Proposition~\ref{prop:switch}, we have that for all $k \geq k_0,$ $A_{k+1} = A_k + \frac{\sqrt{n(n+\sigma A_k)}}{2R'}.$ Thus, we can use Lemma~\ref{lemma:es-seq-growth} to conclude that after iteration $k_0,$ the growth of $A_k$ is quadratic. However, to obtain tighter constants, we will derive a slightly tighter bound in the following proposition.

\begin{proposition}\label{prop:vrpda-quadratic-growth-A_k}
Let $k_0, B_{n, \sigma, R'}$ be defined as in Proposition~\ref{prop:vr-burn-out}. Then, for all $k > k_0,$ $A_k \geq c(k - k_0 + n-1)^2,$ where $c = \frac{(n-1)^2\sigma}{(4R')^2n}.$ 
\end{proposition}
\begin{proof}
We prove the proposition by induction on $k_0$. Observe that $B_{n, \sigma, R'} \geq \frac{\sigma n(n-1)}{2R'}.$ Applying Proposition~\ref{prop:vr-burn-out}, we have that 
\[
A_{k_0} \geq \frac{n-1}{2R'}B_{n, \sigma, R'} \ge \frac{(n-1)^2 n \sigma}{4 (R')^2} = 4cn^2 \ge c(n-1)^2,
\]
so the claim holds for $k=k_0$. Now assume that the claim holds for $k \geq k_0$ and consider iteration $k + 1.$ 
We have
\begin{align*}
    A_{k+1} &= A_k + \frac{\sqrt{n^2 + n\sigma A_k}}{2R'}\\
    &> c(k - k_0 + n-1)^2 + \frac{\sqrt{nc \sigma}}{2R'}(k- k_0 + n - 1). 
\end{align*}
Let us argue that $\frac{\sqrt{nc \sigma}}{2R'}(k- k_0 + n - 1) \geq 2c(k - k_0 + n).$ 
We note first that
\[
\frac{\sqrt{nc\sigma}}{2R'} = \frac{\sqrt{n \sigma}}{2R'} \frac{(n-1)\sqrt{\sigma}}{4R'\sqrt{n}} = 
\frac{(n-1)\sigma}{8 (R')^2} = \frac{n}{n-1} 2c \implies 2c = \left( 1-\frac{1}{n} \right) \frac{\sqrt{nc\sigma}}{2R'}.
\]
We then have 
\begin{align*}
    \frac{\sqrt{nc \sigma}}{2R'}(k- k_0 + n - 1) - 2c(k - k_0 + n) & = \Big(\frac{\sqrt{nc \sigma}}{2R'} - 2c\Big) (k - k_0 + n) - \frac{\sqrt{nc \sigma}}{2R'} \\
&     \geq  \frac{\sqrt{nc \sigma}}{2R'} \frac{1}{n} (k-k_0+n) - \frac{\sqrt{nc \sigma}}{2R'} \\
   & \geq  0,
\end{align*}
where we have used $k \geq k_0$.
Hence, we have that:
\[
    A_{k+1} > c(k - k_0 + n-1)^2 + 2 c (k - k_0 + n)
    = c(k - k_0 + n)^2 + 1 > c(k - k_0 + n)^2,
\]
establishing the inductive step and proving the claim.
\end{proof}
Proposition~\ref{prop:vrpda-quadratic-growth-A_k} is mainly useful when $\sigma$ is not too small. 
For small or zero values of $\sigma,$ however, we can show that after $k_0$ iterations the growth of $A_k$ is at least a linear function of $k$, as follows. 

\begin{proposition}\label{prop:vr-linear-grwoth}
Let $K_0 = \lceil \frac{\log(n)}{\log(n) - \log(n-1)}\rceil$, $n \geq 2$. 
Then, for all $k \geq K_0,$ we have that $A_k \geq \frac{n(k - K_0 + n -1)}{2R'}.$   
\end{proposition}
\begin{proof}
 Since $K_0 \leq k_0,$ we have by Proposition~\ref{prop:vr-burn-out} that $A_{K_0} \geq \frac{n(n-1)}{2 R'}$ and $a_{K_0} \geq \frac{n-1}{2R'}.$ As $a_k = \min \big\{ \big(1+\frac{1}{n-1} \big) a_{k-1}, \frac{\sqrt{n(n+\sigma A_{k-1})}}{2R'} \big\}$ and $\sigma \geq 0,$ for all $k \geq 3,$  we have that $a_k \geq  \frac{n}{2R'}$ for all $k \geq K_0 + 1,$ leading to the claimed bound on $A_k.$ 
\end{proof}

We can now combine Propositions~\ref{prop:vr-burn-out}-\ref{prop:vr-linear-grwoth} to obtain a lower bound on $A_k,$ as summarized in the following lemma. Its proof is a direct consequence of Propositions~\ref{prop:vr-burn-out}-\ref{prop:vr-linear-grwoth}, and is thus omitted.

\begin{lemma}\label{lemma:vr-A_k-growth}
Let sequences $\{a_k\}_{k \geq 1}$, $\{A_k\}_{k \geq 1}$ be defined by Eq.\eqref{eq:akdef}. Then:
\[
A_k \geq \max \left\{\frac{n-1}{2R'}\Big(1 + \frac{1}{n-1}\Big)^k \mathds{1}_{k \leq k_0}, \;
     \frac{(n-1)^2 \sigma}{(4R')^2 n}(k-k_0 + n-1)^2 \mathds{1}_{k \geq k_0},
     \frac{n(k - K_0 + n - 1)}{2R'}\mathds{1}_{k \geq K_0}\right\},
\]
where $\mathds{1}$ denotes the indicator function and
\begin{align*}
& K_0 = \left\lceil \frac{\log n}{\log n - \log(n-1)} \right\rceil, \quad
k_0 = \left\lceil \frac{\log B_{n,\sigma,R'}}{\log n - \log (n-1)} \right\rceil, \\
& B_{n, \sigma, R'} = \frac{\sigma n (n-1)}{4 R'} \left[ 1+ \sqrt{1+ \left(\frac{4R'}{\sigma(n-1)}\right)^2
} 
\right] \geq n\max\left\{1, \frac{\sigma (n-1)}{2 R'}\right\}.
\end{align*}
\end{lemma}

\section{Efficient Implementation of VRPDA$^2$~and Other Computational Considerations}\label{appx:comp-considerations}

\begin{algorithm}[t!]
\caption{Variance Reduction via Primal-Dual Accelerated Dual Averaging (\vrpda~, Implementation Version) }\label{alg:vrpda-implementable}
\begin{algorithmic}[1]
\STATE \textbf{Input: } $(\vx_0, \vy_0)\in\gX\times \gY, (\vu, \vv)\in \gX \times \gY.$
\STATE $a_0 = A_0 = 0, \tilde{a}_1 = \frac{1}{2R'}, \tilde{\vp}_1= - \mB \vx_0.$
\STATE $\vy_1 = \prox_{\tilde{a}_1 g^*}(\vy_0 - \tilde{a}_1\tilde{\vp}_1).$
\STATE $\vz_1 = \mB^T\vy_1$.
\STATE $\vx_1 =  \prox_{\tilde{a}_1 \ell}(\vx_0 - \tilde{a}_1 \vz_1).$
\STATE $a_1 = A_1 = n \tilde{a}_1, a_2 = \frac{1}{n-1}a_1, A_2 = A_1 + a_2.$
\STATE $\vp_1 = a_1 \tilde{\vp}_1, \vq_1 = a_1 \vz_1, \vr_1 = \frac{1}{n}{a}_1\mathbf{1}.$
\FOR{$k = 2,3,\ldots, K$}
\STATE $\bar{\vx}_{k-1} =  \vx_{k-1} + \frac{a_{k-1}}{a_{k}}(\vx_{k-1} - \vx_{k-2}).$
\STATE Pick $j_k$ uniformly at random in $[n].$
\STATE $p_{k,i} = \begin{cases}
p_{k-1,i}, &  i \neq j_k \\
p_{k-1,i} - a_k\vb_{j_k}^T \bar{\vx}_{k-1}, & i = j_k
\end{cases}$,  \quad\quad\,   $r_{k,i} = \begin{cases}
r_{k-1,i}, &  i \neq j_k \\
r_{k-1,i} + a_k , & i = j_k
\end{cases}$.
\STATE $ y_{k,i} = 
\begin{cases}
y_{k-1,i}, & i \neq j_k \\
\prox_{\frac{1}{n} r_{k,j_k} g^*_{j_k}}(y_{0,j_k} -\frac{1}{n} p_{k,j_k}), &  i = j_k\\
\end{cases} $.
\STATE $\vq_k = \vq_{k-1} + a_k (\vz_{k-1} + (y_{k, j_k} - y_{k-1, j_k})\vb_{j_k}).$
\STATE $\vx_k = \prox_{\frac{1}{n}A_k\ell}(\vx_0 - \frac{1}{n}\vq_k).$ 
\STATE $\vz_k = \vz_{k-1} + \frac{1}{n}(y_{k, j_k} - y_{k-1, j_k})\vb_{j_k}.$

\STATE $a_{k+1} = \min \Big( \big(1+\frac{1}{n-1} \big) a_{k}, \frac{\sqrt{n(n+\sigma A_{k})}}{2R'} \Big), A_{k+1} = A_k + a_{k+1}$.
\ENDFOR
\STATE  \textbf{return } $\vx_K$ or $\tilde{\vx}_K := \frac{1}{A_K}\sum_{k=1}^K a_k \vx_k$.
\end{algorithmic}
\end{algorithm}

We discuss here an equivalent form of Algorithm~\ref{alg:vrpda} for which the aspects needed for efficient implementation are treated more transparently. To implement Algorithm~\ref{alg:vrpda}, we use the proximal operator $\prox_{\tau f}(\vx_0) = \argmin_{\vx}\{\tau f(\vx) + \frac{1}{2}\|\vx - \vx_0\|^2\}$ for a scalar $\tau >0$, a convex function $f$ and a given $\vx_.$
This implementation version (shown here as Algorithm~\ref{alg:vrpda-implementable}), 
maintains several additional vectors that are needed to keep track of coefficients in the ``$\argmin$" Steps 11 and 12 of Algorithm~\ref{alg:vrpda}.
In particular, we maintain a vector $\vp_k \in \R^n$ that contains the coefficients of the linear term in $\vy$ in the function $\psi_k(\cdot)$, a vector $\vr_k \in \R^n$ that contains coefficients of $g^*_{j}(\cdot)$  in $\psi_k(\cdot)$, and a vector $\vq_k \in \R^d$ that is the coefficient of the linear term in  $\vx$ in the function $\phi_k(n\cdot)$.
Each of these vectors generally must be stored in full, and all are initialized in Step 7 of Algorithm~\ref{alg:vrpda-implementable}.
In Step 11 of Algorithm~\ref{alg:vrpda-implementable}, only one component---the $j_k$ component---of $\vp_k$ and $\vr_k$ needs to be updated. 
For $\vr_k$, the cost of update is one scalar addition, whereas for $\vp_k$ it requires computation of the inner product $\vb_{j_k}^T \bar{\vx}_{k-1}$, which costs $O(d)$ scalar operations if $\vb_{j_k}$ is dense but potentially much less for sparse $\mB$. 
The update of $\vq_k$ (Step 14 of Algorithm~\ref{alg:vrpda-implementable}) requires addition of scalar multiples of two vectors ($\vz_{k-1}$ and $\vb_{j_k}$), also at a cost of $O(d)$ operations in general. (Note that for the latter operation, savings are possible if $\mB$ is sparse, but since $\vz_{k-1}$ is generally dense, the overall cost will still be $O(d)$.)
In discussing the original Algorithm~\ref{alg:vrpda}, we noted that the $\argmin$ operation over $\vy$ in Step 11 resulted in only a single component (component $j_k$) being different between $\vy_{k-1}$ and $\vy_k$. 
We make this fact explicit in Step 12 of Algorithm~\ref{alg:vrpda-implementable}, where we show precisely the prox-operation that needs to be performed to recover the scalar $y_{k,j_k}$.

To summarize, each iteration of Algorithm~\ref{alg:vrpda-implementable} requires several vector operations with the (possibly sparse) vector $\vb_{j_k} \in \R^d$, several other vector operations of cost  $O(d)$, a prox operation involving $\ell(\vx)$, and a scalar prox operation involving $g_{j_k}^*$. There are no operations of cost $O(n)$.

Initialization of Algorithm~\ref{alg:vrpda-implementable} involves significant costs, including one matrix-vector product each involving $\mB$ and $\mB^T$, one prox operation involving $\ell(\cdot)$, a prox operation involving $g^*(\cdot)$ (which can be implemented as $n$ scalar prox operations involving $g_1^*, g_2^*, \dotsc, g_n^*$ in turn), and several vector operations of cost $O(d+n)$. However, this cost is absorbed by the overall (amortized) computational cost of the algorithm, which adds to $O(nd \log(\min\{n, 1/\epsilon\}) + \frac{d}{\epsilon})$ for the general convex case and   $O(nd \log(\min\{n, 1/\epsilon\}) + \frac{d}{\sqrt{\sigma\epsilon}})$ for the strongly convex case, as stated in Theorem~\ref{thm:vr} and in Table~\ref{tb:result}.

\end{document}

%% file: math_commands.tex
\usepackage{mathrsfs}
\usepackage{amsmath,amsfonts,amsthm}
\usepackage{bm}

\usepackage[colorinlistoftodos]{todonotes}
\usepackage{thmtools}
\usepackage{thm-restate}
\usepackage{dsfont}

\ificml
\usepackage{balance}
\usepackage{threeparttable}
\usepackage{xcolor}
\definecolor{myblue}{rgb}{0,0.08,0.45}
\usepackage[colorlinks,urlcolor=myblue,citecolor=myblue,linkcolor=myblue]{hyperref}

\else 
\usepackage{threeparttable}
\usepackage[format=hang,font={small}]{caption}
\usepackage{algorithm,algorithmic}
\usepackage[algo2e,ruled,vlined]{algorithm2e}
\usepackage[pdftex,bookmarksnumbered,bookmarksopen,
colorlinks,citecolor=blue,linkcolor=blue,urlcolor=blue]{hyperref}
\usepackage{natbib}
\setcitestyle{sort,comma,authoryear,round}
\fi

 \usepackage{color}

\usepackage{psfrag}
\usepackage{url}            %
\usepackage{booktabs}       %
\usepackage{amsfonts}       %
\usepackage{nicefrac}       %
\usepackage{microtype}      %

\usepackage{amssymb}
\usepackage{graphicx}
\usepackage{setspace}
\usepackage{subfigure}
\usepackage{enumitem}

\usepackage{multicol}
\usepackage{multirow}

\newtheorem{lemma}{Lemma}

\newtheorem{remark}{Remark}
\newtheorem{proposition}{Proposition}
\newtheorem{assumption}{Assumption}

\newtheorem{definition}{Definition}

\def\1{\bm{1}}

\def\vzero{{\bm{0}}}

\def\vb{{\bm{b}}}

\def\ve{{\bm{e}}}

\def\vg{{\bm{g}}}

\def\vp{{\bm{p}}}
\def\vq{{\bm{q}}}
\def\vr{{\bm{r}}}

\def\vu{{\bm{u}}}
\def\vv{{\bm{v}}}

\def\vx{{\bm{x}}}
\def\vy{{\bm{y}}}
\def\vz{{\bm{z}}}

\def\vxh{{\hat{\bm{x}}}}
\def\vxb{{\bar{\bm{x}}}}

\def\vxt{{\Tilde{\bm{x}}}}
\def\vyt{{\Tilde{\bm{y}}}}

\def\mB{{\bm{B}}}

\DeclareMathAlphabet{\mathsfit}{\encodingdefault}{\sfdefault}{m}{sl}
\SetMathAlphabet{\mathsfit}{bold}{\encodingdefault}{\sfdefault}{bx}{n}

\def\gB{{\mathcal{B}}}

\def\gF{{\mathcal{F}}}

\def\gX{{\mathcal{X}}}
\def\gY{{\mathcal{Y}}}

\def\sR{{\mathbb{R}}}

\newcommand{\E}{\mathbb{E}}

\newcommand{\R}{\mathbb{R}}

\newcommand{\Gap}{\mathrm{Gap}}

\DeclareMathOperator*{\argmax}{arg\,max}
\DeclareMathOperator*{\argmin}{arg\,min}

\newcommand{\vrpda}{\textsc{vrpda}$^2$}
\newcommand{\pda}{\textsc{pda}$^2$}
\newcommand{\pdhg}{\textsc{pdhg}}
\newcommand{\spdhg}{\textsc{spdhg}}
\newcommand{\pure}{\textsc{pure\_cd}}
\newcommand{\vrada}{\textsc{vrada}}
\newcommand{\prox}{\text{prox}}

\newcommand{\innp}[1]{\left\langle #1 \right\rangle}

%% file: VRPDA2.bbl
\begin{thebibliography}{45}
\providecommand{\natexlab}[1]{#1}
\providecommand{\url}[1]{\texttt{#1}}
\expandafter\ifx\csname urlstyle\endcsname\relax
  \providecommand{\doi}[1]{doi: #1}\else
  \providecommand{\doi}{doi: \begingroup \urlstyle{rm}\Url}\fi

\bibitem[LIB()]{LIBSVM}
{LIBSVM} {L}ibrary.
\newblock \url{https://www.csie.ntu.edu.tw/~cjlin/libsvm/index.html}.
\newblock Accessed: Feb.~3, 2020.

\bibitem[Alacaoglu et~al.(2017)Alacaoglu, Dinh, Fercoq, and
  Cevher]{alacaoglu2017smooth}
Ahmet Alacaoglu, Quoc~Tran Dinh, Olivier Fercoq, and Volkan Cevher.
\newblock Smooth primal-dual coordinate descent algorithms for nonsmooth convex
  optimization.
\newblock In \emph{Proc.~NIPS'17}, 2017.

\bibitem[Alacaoglu et~al.(2020)Alacaoglu, Fercoq, and
  Cevher]{alacaoglu2020random}
Ahmet Alacaoglu, Olivier Fercoq, and Volkan Cevher.
\newblock Random extrapolation for primal-dual coordinate descent.
\newblock In \emph{Proc.~ICML'20}, 2020.

\bibitem[Allen-Zhu(2017)]{allen2017katyusha}
Zeyuan Allen-Zhu.
\newblock Katyusha: The first direct acceleration of stochastic gradient
  methods.
\newblock In \emph{Proc.~ACM STOC'17}, 2017.

\bibitem[Allen-Zhu and Hazan(2016)]{allen2016optimal}
Zeyuan Allen-Zhu and Elad Hazan.
\newblock Optimal black-box reductions between optimization objectives.
\newblock In \emph{Proc.~NIPS'16}, 2016.

\bibitem[Allen-Zhu and Orecchia(2017)]{allen2017linear}
Zeyuan Allen-Zhu and Lorenzo Orecchia.
\newblock Linear coupling: An ultimate unification of gradient and mirror
  descent.
\newblock In \emph{Proc.~ITCS'17}, 2017.

\bibitem[Belloni et~al.(2011)Belloni, Chernozhukov, and
  Wang]{belloni2011square}
Alexandre Belloni, Victor Chernozhukov, and Lie Wang.
\newblock Square-root lasso: Pivotal recovery of sparse signals via conic
  programming.
\newblock \emph{Biometrika}, 98\penalty0 (4):\penalty0 791--806, 2011.

\bibitem[Blanchet et~al.(2019)Blanchet, Kang, and Murthy]{blanchet2019robust}
Jose Blanchet, Yang Kang, and Karthyek Murthy.
\newblock Robust {W}asserstein profile inference and applications to machine
  learning.
\newblock \emph{Journal of Applied Probability}, 56\penalty0 (3):\penalty0
  830--857, 2019.

\bibitem[Carmon et~al.(2019)Carmon, Jin, Sidford, and Tian]{carmon2019variance}
Yair Carmon, Yujia Jin, Aaron Sidford, and Kevin Tian.
\newblock Variance reduction for matrix games.
\newblock In \emph{Proc.~NeurIPS'19}, 2019.

\bibitem[Chambolle and Pock(2011)]{chambolle2011first}
Antonin Chambolle and Thomas Pock.
\newblock A first-order primal-dual algorithm for convex problems with
  applications to imaging.
\newblock \emph{Journal of mathematical imaging and vision}, 40\penalty0
  (1):\penalty0 120--145, 2011.

\bibitem[Chambolle et~al.(2018)Chambolle, Ehrhardt, Richt{\'a}rik, and
  Schonlieb]{chambolle2018stochastic}
Antonin Chambolle, Matthias~J Ehrhardt, Peter Richt{\'a}rik, and Carola-Bibiane
  Schonlieb.
\newblock Stochastic primal-dual hybrid gradient algorithm with arbitrary
  sampling and imaging applications.
\newblock \emph{SIAM Journal on Optimization}, 28\penalty0 (4):\penalty0
  2783--2808, 2018.

\bibitem[Chen et~al.(2017)Chen, Lan, and Ouyang]{chen2017accelerated}
Yunmei Chen, Guanghui Lan, and Yuyuan Ouyang.
\newblock Accelerated schemes for a class of variational inequalities.
\newblock \emph{Mathematical Programming}, 165\penalty0 (1):\penalty0 113--149,
  2017.

\bibitem[Dang and Lan(2014)]{dang2014randomized}
Cong Dang and Guanghui Lan.
\newblock Randomized first-order methods for saddle point optimization.
\newblock \emph{arXiv preprint arXiv:1409.8625}, 2014.

\bibitem[Defazio et~al.(2014)Defazio, Bach, and
  Lacoste-Julien]{defazio2014saga}
Aaron Defazio, Francis Bach, and Simon Lacoste-Julien.
\newblock {SAGA}: A fast incremental gradient method with support for
  non-strongly convex composite objectives.
\newblock In \emph{Proc.~NIPS'14}, 2014.

\bibitem[Devraj and Chen(2019)]{devraj2019stochastic}
Adithya~M Devraj and Jianshu Chen.
\newblock Stochastic variance reduced primal dual algorithms for empirical
  composition optimization.
\newblock In \emph{Proc.~NeurIPS'19}, 2019.

\bibitem[Diakonikolas(2020)]{diakonikolas2020halpern}
Jelena Diakonikolas.
\newblock Halpern iteration for near-optimal and parameter-free monotone
  inclusion and strong solutions to variational inequalities.
\newblock In \emph{Proc.~COLT'20}, 2020.

\bibitem[Diakonikolas et~al.(2020)Diakonikolas, Daskalakis, and
  Jordan]{diakonikolas2020efficient}
Jelena Diakonikolas, Constantinos Daskalakis, and Michael~I Jordan.
\newblock Efficient methods for structured nonconvex-nonconcave min-max
  optimization.
\newblock \emph{arXiv preprint arXiv:2011.00364}, 2020.

\bibitem[Fercoq and Bianchi(2019)]{fercoq2019coordinate}
Olivier Fercoq and Pascal Bianchi.
\newblock A coordinate-descent primal-dual algorithm with large step size and
  possibly nonseparable functions.
\newblock \emph{SIAM Journal on Optimization}, 29\penalty0 (1):\penalty0
  100--134, 2019.

\bibitem[Hannah et~al.(2018)Hannah, Liu, O'Connor, and Yin]{hannah2018breaking}
Robert Hannah, Yanli Liu, Daniel O'Connor, and Wotao Yin.
\newblock Breaking the span assumption yields fast finite-sum minimization.
\newblock In \emph{Proc.~NeurIPS'18}, 2018.

\bibitem[Hazan et~al.(2016)]{hazan2016introduction}
Elad Hazan et~al.
\newblock Introduction to online convex optimization.
\newblock \emph{Foundations and Trends{\textregistered} in Optimization},
  2\penalty0 (3-4):\penalty0 157--325, 2016.

\bibitem[Johnson and Zhang(2013)]{johnson2013accelerating}
Rie Johnson and Tong Zhang.
\newblock Accelerating stochastic gradient descent using predictive variance
  reduction.
\newblock In \emph{Proc.~NIPS'13}, 2013.

\bibitem[Lan et~al.(2019)Lan, Li, and Zhou]{lan2019unified}
Guanghui Lan, Zhize Li, and Yi~Zhou.
\newblock A unified variance-reduced accelerated gradient method for convex
  optimization.
\newblock In \emph{Proc.~NeurIPS'19}, 2019.

\bibitem[Latafat et~al.(2019)Latafat, Freris, and Patrinos]{latafat2019new}
Puya Latafat, Nikolaos~M Freris, and Panagiotis Patrinos.
\newblock A new randomized block-coordinate primal-dual proximal algorithm for
  distributed optimization.
\newblock \emph{IEEE Transactions on Automatic Control}, 64\penalty0
  (10):\penalty0 4050--4065, 2019.

\bibitem[Lei et~al.(2019)Lei, Zhuo, Caramanis, Dhillon, and
  Dimakis]{lei2019primal}
Qi~Lei, Jiacheng Zhuo, Constantine Caramanis, Inderjit~S Dhillon, and
  Alexandros~G Dimakis.
\newblock Primal-dual block generalized {F}rank-{W}olfe.
\newblock \emph{Proc.~NeurIPS'19}, 2019.

\bibitem[Lin et~al.(2014)Lin, Lu, and Xiao]{lin2014accelerated}
Qihang Lin, Zhaosong Lu, and Lin Xiao.
\newblock An accelerated proximal coordinate gradient method.
\newblock \emph{Proc.~NIPS'14}, 2014.

\bibitem[Nemirovski(2004)]{nemirovski2004prox}
Arkadi Nemirovski.
\newblock Prox-method with rate of convergence ${O} (1/t)$ for variational
  inequalities with lipschitz continuous monotone operators and smooth
  convex-concave saddle point problems.
\newblock \emph{SIAM Journal on Optimization}, 15\penalty0 (1):\penalty0
  229--251, 2004.

\bibitem[Nesterov(2005{\natexlab{a}})]{nesterov2005excessive}
Yu~Nesterov.
\newblock Excessive gap technique in nonsmooth convex minimization.
\newblock \emph{SIAM Journal on Optimization}, 16\penalty0 (1):\penalty0
  235--249, 2005{\natexlab{a}}.

\bibitem[Nesterov(2005{\natexlab{b}})]{nesterov2005smooth}
Yu~Nesterov.
\newblock Smooth minimization of non-smooth functions.
\newblock \emph{Mathematical programming}, 103\penalty0 (1):\penalty0 127--152,
  2005{\natexlab{b}}.

\bibitem[Nesterov(2007)]{nesterov2007dual}
Yurii Nesterov.
\newblock Dual extrapolation and its applications to solving variational
  inequalities and related problems.
\newblock \emph{Mathematical Programming}, 109\penalty0 (2-3):\penalty0
  319--344, 2007.

\bibitem[Nesterov(2015)]{nesterov2015universal}
Yurii Nesterov.
\newblock Universal gradient methods for convex optimization problems.
\newblock \emph{Mathematical Programming}, 152\penalty0 (1-2):\penalty0
  381--404, 2015.

\bibitem[Nesterov(2018)]{nesterov2018lectures}
Yurii Nesterov.
\newblock \emph{Lectures on convex optimization}, volume 137.
\newblock Springer, 2018.

\bibitem[Ouyang and Xu(2019)]{ouyang2018lower}
Yuyuan Ouyang and Yangyang Xu.
\newblock Lower complexity bounds of first-order methods for convex-concave
  bilinear saddle-point problems.
\newblock \emph{Mathematical Programming}, pages 1--35, 2019.

\bibitem[Ravikumar et~al.(2010)Ravikumar, Wainwright, Lafferty,
  et~al.]{ravikumar2010high}
Pradeep Ravikumar, Martin~J Wainwright, John~D Lafferty, et~al.
\newblock High-dimensional {I}sing model selection using $\ell_1$-regularized
  logistic regression.
\newblock \emph{The Annals of Statistics}, 38\penalty0 (3):\penalty0
  1287--1319, 2010.

\bibitem[Rockafellar and Wets(2009)]{rockafellar2009variational}
R~Tyrrell Rockafellar and Roger J-B Wets.
\newblock \emph{Variational analysis}, volume 317.
\newblock Springer Science \& Business Media, 2009.

\bibitem[Roux et~al.(2012)Roux, Schmidt, and Bach]{roux2012stochastic}
Nicolas Roux, Mark Schmidt, and Francis Bach.
\newblock A stochastic gradient method with an exponential convergence rate for
  finite training sets.
\newblock \emph{Proc.~NIPS'12}, 2012.

\bibitem[Song et~al.(2020{\natexlab{a}})Song, Jiang, and Ma]{song2020variance}
Chaobing Song, Yong Jiang, and Yi~Ma.
\newblock Variance reduction via accelerated dual averaging for finite-sum
  optimization.
\newblock \emph{Proc.~NeurIPS'20}, 2020{\natexlab{a}}.

\bibitem[Song et~al.(2020{\natexlab{b}})Song, Zhou, Zhou, Jiang, and
  Ma]{song2020optimistic}
Chaobing Song, Zhengyuan Zhou, Yichao Zhou, Yong Jiang, and Yi~Ma.
\newblock Optimistic dual extrapolation for coherent non-monotone variational
  inequalities.
\newblock \emph{Proc.~NeurIPS'20}, 2020{\natexlab{b}}.

\bibitem[Tan et~al.(2018)Tan, Zhang, Ma, and Liu]{tan2018stochastic}
Conghui Tan, Tong Zhang, Shiqian Ma, and Ji~Liu.
\newblock Stochastic primal-dual method for empirical risk minimization with $o
  (1)$ per-iteration complexity.
\newblock In \emph{Proc.~NeurIPS'18}, 2018.

\bibitem[Tibshirani(1996)]{tibshirani1996regression}
Robert Tibshirani.
\newblock Regression shrinkage and selection via the lasso.
\newblock \emph{Journal of the Royal Statistical Society: Series B
  (Methodological)}, 58\penalty0 (1):\penalty0 267--288, 1996.

\bibitem[Tran-Dinh et~al.(2018)Tran-Dinh, Fercoq, and Cevher]{tran2018smooth}
Quoc Tran-Dinh, Olivier Fercoq, and Volkan Cevher.
\newblock A smooth primal-dual optimization framework for nonsmooth composite
  convex minimization.
\newblock \emph{SIAM Journal on Optimization}, 28\penalty0 (1):\penalty0
  96--134, 2018.

\bibitem[Woodworth and Srebro(2016)]{woodworth2016tight}
Blake~E Woodworth and Nati Srebro.
\newblock Tight complexity bounds for optimizing composite objectives.
\newblock \emph{Proc.~NIPS'16}, 2016.

\bibitem[Xiao(2010)]{xiao2010dual}
Lin Xiao.
\newblock Dual averaging methods for regularized stochastic learning and online
  optimization.
\newblock \emph{Journal of Machine Learning Research}, 11\penalty0
  (Oct):\penalty0 2543--2596, 2010.

\bibitem[Zhang and Lin(2015)]{zhang2015stochastic}
Yuchen Zhang and Xiao Lin.
\newblock Stochastic primal-dual coordinate method for regularized empirical
  risk minimization.
\newblock In \emph{Proc.~ICML'15}, 2015.

\bibitem[Zhou et~al.(2018)Zhou, Shang, and Cheng]{zhou2018simple}
Kaiwen Zhou, Fanhua Shang, and James Cheng.
\newblock A simple stochastic variance reduced algorithm with fast convergence
  rates.
\newblock In \emph{Proc.~ICML'18}, 2018.

\bibitem[Zou and Hastie(2005)]{zou2005regularization}
Hui Zou and Trevor Hastie.
\newblock Regularization and variable selection via the elastic net.
\newblock \emph{Journal of the royal statistical society: series B (statistical
  methodology)}, 67\penalty0 (2):\penalty0 301--320, 2005.

\end{thebibliography}
